\documentclass[11pt,a4paper, reqno]{article}
\usepackage[utf8]{inputenc}
\usepackage{amsfonts}
\usepackage{amssymb}
\usepackage{enumerate}
\usepackage{graphicx}
\usepackage{mathtools}
\usepackage{amsmath, amsthm}
\usepackage{amssymb,amsthm,amsfonts,amsbsy,latexsym,dsfont}

\usepackage{tikz}
\usetikzlibrary{calc,intersections,through,backgrounds,shapes.geometric}
\usetikzlibrary{graphs}

\usepackage{color}
\usepackage[affil-it]{authblk}
\usepackage[T1]{fontenc}
\usepackage[sort,numbers]{natbib}
\usepackage{bbm}
\usepackage[a4paper, left = 1in, right = 1in, top = 1in, bottom = 1in]{geometry}
\usepackage{amsmath}
\usepackage{amsfonts}
\usepackage{amssymb}
\usepackage{bbm}
\usepackage{mathrsfs}
\usepackage[utf8]{inputenc}
\usepackage[english]{babel}
\usepackage{hyperref}
\usepackage{relsize}
\usepackage{accents}
\usepackage[noabbrev]{cleveref}
\usepackage{appendix}
\usepackage{lipsum}
\usepackage{placeins}
\usepackage[color = white]{todonotes}

\usepackage{caption}
\usepackage{subcaption}
\usepackage{color}

\usepackage[mode = buildnew]{standalone}

\usepackage{enumerate}
\newenvironment{enumeratei}{\begin{enumerate}[\upshape (i)]}{\end{enumerate}}

\newcommand\blfootnote[1]{%
  \begingroup
  \renewcommand\thefootnote{}\footnote{#1}%
  \addtocounter{footnote}{-1}%
  \endgroup
}

\makeatletter
  \@addtoreset{chapter}{part}
  \@addtoreset{@ppsaveapp}{part}
\makeatother

\usepackage{palatino}

\makeatletter
\newsavebox\myboxA
\newsavebox\myboxB
\newlength\mylenA

\newcommand*\xoverline[2][0.75]{%
    \sbox{\myboxA}{$\m@th#2$}%
    \setbox\myboxB\null
    \ht\myboxB=\ht\myboxA%
    \dp\myboxB=\dp\myboxA%
    \wd\myboxB=#1\wd\myboxA
    \sbox\myboxB{$\m@th\overline{\copy\myboxB}$}
    \setlength\mylenA{\the\wd\myboxA}
    \addtolength\mylenA{-\the\wd\myboxB}%
    \ifdim\wd\myboxB<\wd\myboxA%
       \rlap{\hskip 0.5\mylenA\usebox\myboxB}{\usebox\myboxA}%
    \else
        \hskip -0.5\mylenA\rlap{\usebox\myboxA}{\hskip 0.5\mylenA\usebox\myboxB}%
    \fi}
\makeatother

\usepackage{hyperref}
\hypersetup{
  colorlinks,
  linkcolor=blue,
  citecolor=black,
  filecolor=black,
  urlcolor=black,
  pdftitle={},
  pdfauthor={S. Dhara, R. van der Hofstad, J. van Leeuwaarden, S. Sen},
  pdfcreator={S. Dhara},
  pdfsubject={},
  pdfkeywords={}
}
\numberwithin{equation}{section}
\usepackage{cleveref}

\def\sss{\scriptscriptstyle}
\newcommand{\ubar}[1]{\underaccent{\bar}{#1}}

\newcommand{\prob}[1]{\ensuremath{\mathbbm{P}\left(#1\right)}}
\newcommand{\expt}[1]{\ensuremath{\mathbbm{E}\left[#1\right]}}
\newcommand{\var}[1]{\ensuremath{\mathrm{Var}\left(#1\right)}}

\newcommand{\floor}[1]{\ensuremath{\left\lfloor #1 \right\rfloor}}
\newcommand{\ind}[1]{\ensuremath{\mathbbm{1}\left\{#1\right\}}}

\newcommand{\dto}{\ensuremath{\xrightarrow{d}}}

\newcommand{\diam}{\ensuremath{\mathrm{diam}}}
\newcommand{\PR}{\ensuremath{\mathbbm{P}}}
\newcommand{\E}{\ensuremath{\mathbbm{E}}}
\newcommand{\R}{\ensuremath{\mathbb{R}}}
\newcommand{\N}{\ensuremath{\mathbb{N}}}
\newcommand{\1}{\ensuremath{\mathbbm{1}}}

\newcommand{\dst}{\ensuremath{\mathrm{d}}}

\newcommand{\shortarrow}{{\sss \downarrow}}

\newcommand{\bld}[1]{\boldsymbol{#1}}
\newcommand{\oP}{o_{\sss \PR}}
\newcommand{\OP}{O_{\sss \PR}}
\newcommand{\thetaP}{\ensuremath{\Theta_{\sss\PR}}}
\newcommand{\e}{\mathrm{e}}
\newcommand{\dif}{\mathrm{d}}
\newcommand{\csi}{\mathscr{C}_{\sss (i)}}
\newcommand{\cs}{\mathscr{C}}
\newcommand{\scl}{\mathrm{scl}}
\newcommand{\visit}{\mathrm{visit}}

\newcommand{\blue}[1]{{\color{black}#1}}
\newcommand{\red}[1]{{\color{black}#1}}

\newcommand{\rCM}{\mathrm{CM}}
\newcommand{\CM}{\mathrm{CM}_n(\bld{d})}
\newcommand{\UM}{\mathrm{UM}_n(\bld{d})}

\newtheorem*{ass*}{Assumption}
\newtheorem*{theorem*}{Theorem}

\newtheorem{theorem}{Theorem}[section]
\newtheorem{algo}{Algorithm}
 \newtheorem{lemma}[theorem]{Lemma}
\newtheorem{proposition}[theorem]{Proposition}
\newtheorem{corollary}[theorem]{Corollary}
\newtheorem{assumption}{Assumption}
\newtheorem{remark}{Remark}
\newtheorem{fact}{Fact}
\newtheorem{defn}{Definition}

\theoremstyle{definition}

\usepackage{environ}
\NewEnviron{eq}{%
\begin{equation}\begin{split}
  \BODY
\end{split}\end{equation}
}

\numberwithin{equation}{section}

\newcommand{\cA}{\mathcal{A}}\newcommand{\cB}{\mathcal{B}}
\newcommand{\cF}{\mathcal{F}}
\newcommand{\cG}{\mathcal{G}}
\newcommand{\cK}{\mathcal{K}}
\newcommand{\cN}{\mathcal{N}}

\newcommand{\cX}{\mathcal{X}}

\newcommand{\fm}{\mathfrak{m}}

\newcommand{\bZ}{\mathbb{Z}}

\newcommand{\rE}{\mathrm{E}}

\newcommand{\sC}{\mathscr{C}}

\DeclareMathOperator{\pr}{\mathbbm{P}}

\newcommand{\weakc}{\stackrel{d}{\longrightarrow}}

\begin{document}
\title{Global lower mass-bound for critical configuration models in the heavy-tailed regime}
\author{Shankar Bhamidi$^1$, Souvik Dhara$^{2,3}$, Remco van der Hofstad$^{4}$,  Sanchayan Sen$^5$}
\maketitle
\blfootnote{\emph{Emails:} 
 \href{mailto:bhamidi@email.unc.edu}{bhamidi@email.unc.edu},
 \href{mailto:sdhara@mit.edu}{sdhara@mit.edu},
 \href{mailto:r.w.v.d.hofstad@tue.nl}{r.w.v.d.hofstad@tue.nl},
 \href{mailto:sanchayan.sen1@gmail.com}{sanchayan.sen1@gmail.com}} 
\blfootnote{$^1$Department of Statistics and Operations Research,  University of North Carolina}
\blfootnote{$^2$ Department of Mathematics, Massachusetts Institute of Technology}
\blfootnote{$^3$ Microsoft Research Lab -- New England}
\blfootnote{$^4$Department of Mathematics and Computer Science, Eindhoven University of Technology}
\blfootnote{$^5$Department of Mathematics, Indian Institute of Science}
\blfootnote{\emph{Acknowledgment}. 
The authors are grateful to two anonymous referees for their careful reading and many comments and suggestions on an earlier version of the paper.
SB was partially supported by NSF grants DMS-1613072, DMS-1606839, and ARO grant W911NF-17-1-0010. 
SD and  RvdH were supported by the Netherlands Organisation for Scientific Research (NWO) through Gravitation Networks grant 024.002.003. 
In addition, RvdH was supported by VICI grant 639.033.806. 
SS has been supported in part by the Infosys foundation, Bangalore, and by MATRICS grant MTR/2019/000745 from SERB.
SD thanks Eindhoven University of Technology, where a major part of this work was done. 
}
\maketitle
\begin{abstract}
We establish the global lower mass-bound property for the largest connected components in the critical window for the configuration model when the degree distribution has an infinite third moment.
The scaling limit of the critical percolation clusters, viewed as measured metric spaces, was established in \cite{BDHS17} with respect to the Gromov-weak topology.
Our result extends those scaling limit results to  the stronger Gromov-Hausdorff-Prokhorov topology under slightly stronger assumptions on the degree distribution.
This implies the distributional convergence of global functionals such as the diameter of the largest critical components.
Further, our result gives a sufficient condition for compactness of the random metric spaces that arise as scaling limits of critical clusters in the heavy-tailed regime. 
\end{abstract}

\section{Introduction}
Any finite, connected graph $\sC$ can be viewed as a metric space with the distance between points given by $a \dst (\cdot,\cdot)$ for some constant $a>0$, where $\dst(\cdot,\cdot)$ is used as a generic notation to denote the  graph-distance (i.e., number of edges in the shortest path between vertices). 
There is a natural probability measure $\mu$ associated to the metric space $(\sC,a \dst)$ given by  $\mu (A) = |A|/|\sC| $ for any $A\subset \sC$, where $|A|$ denotes the number of vertices in $A$. 
We denote this metric measure space by $(\sC,\red{a})$.
Fix any $\delta>0$ and define the $\delta$-lower mass of $(\sC,\red{a})$ by 
\begin{eq}\label{eq:defn:GLM}
	\fm(\delta):=  \frac{\inf_{u\in \sC}\big|\{v\in \sC: a \dst(v,u) \leq \delta\}\big|}{|\sC|}.
\end{eq}
Thus, $\fm (\delta)$ is the least \emph{mass} in any $\delta$-neighborhood  of a vertex in  $(\sC,\red{a})$.
For a sequence $(\sC_n,a_n)_{n\geq 1}$ of graphs viewed as metric measure spaces, the global lower mass-bound property is defined as follows:
\begin{defn}[Global lower mass-bound property \cite{ALW16-2}]\normalfont 
	For $\delta>0$, let $\fm_n(\delta)$ denote the $\delta$-lower mass of $(\sC_n,a_n)$.
	Then $(\sC_n,a_n)_{n\geq 1}$ is said to satisfy the global lower mass-bound property if and only if $\sup_{n\geq 1} \fm_n(\delta)^{-1}<\infty$ for any $\delta >0$.
	When $(\sC_n)_{n\geq 1}$ is a collection of random graphs, $(\sC_n,a_n)_{n\geq 1}$ is said to satisfy the global lower mass-bound property if and only if $(\fm_n(\delta)^{-1})_{n\geq 1}$ is a tight sequence of random variables for any $\delta >0$.
\end{defn}
The aim of this paper is to prove the global lower mass-bound property for largest connected components of random graphs with given degrees (configuration model) at criticality, when the third moment of the empirical degree distribution tends to infinity (Theorem~\ref{thm:gml-bound}).  
Informally speaking, the global lower mass-bound property ensures that all the small neighborhoods of vertices in the `large' critical component have mass bounded away from zero, so that the component does not have any \emph{light spots} and the total mass is well-distributed over the whole component.
This has several interesting consequences in the theory of critical random graphs. 
Our main motivation comes from the work of Athreya, L\"ohr, and Winter~\cite{ALW16-2}, who have shown that the global lower mass-bound property can be used to prove Gromov-Hausdorff-Prokhorov (GHP) convergence of random metric spaces. 
In a previous paper \cite{BDHS17}, we have  studied the critical percolation clusters for the configuration model in the heavy-tailed universality class. 
We have proved that the ordered vector of components converges in distribution to suitable random objects in the Gromov-weak topology. 
The global lower mass-bound in this paper shows that the result of \cite{BDHS17} in fact holds with respect to the stronger GHP-topology.
One motivating reason for proving the GHP-convergence is that it yields the scaling limit of  global functionals like the diameter of large critical components. 
Finding the scaling limit for the diameter of critical components is a daunting task even for the Erd\H{o}s-R\'enyi random graph.
Nachmias and Peres~\cite{NP08} estimated the tail probabilities of the diameter, but showing a distributional convergence result was a difficult question, until the seminal paper by Addario-Berry,  Broutin and Goldschmidt~\cite{ABG09} that proved the GHP-convergence for critical Erd\H{o}s-R\'enyi random graphs. 
As a corollary of Theorem~\ref{thm:gml-bound}, we also get distributional convergence of the suitably rescaled diameter of the critical percolation clusters in the heavy-tailed regime (Theorem~\ref{thm:GHP}), where the scaling limit and exponents turn out to be different than those for the Erd\H{o}s-R\'enyi case.

We will further discuss the applications and the scope of this work as well as its technical contributions after stating our results in Section~\ref{sec:discussion}.
We start by defining the configuration model and state the precise assumptions.

\subsection{The configuration model}
Consider a non-increasing sequence of degrees $\boldsymbol{d} = ( d_i )_{i \in [n]}$ such that $\ell_n = \sum_{i \in [n]}d_i$ is even. For notational convenience, we suppress the dependence of the degree sequence on $n$. The configuration model on $n$ vertices having degree sequence $\boldsymbol{d}$ is constructed as follows \cite{B80,BC78}:
\begin{itemize}
	\item[] Equip vertex $j$ with $d_{j}$ stubs, or \emph{half-edges}. Two half-edges create an edge once they are paired. Therefore, initially we have $\ell_n=\sum_{i \in [n]}d_i$ half-edges. Pick any one half-edge and pair it with another uniformly chosen half-edge from the remaining unpaired half-edges and keep repeating the above procedure until all the unpaired half-edges are exhausted. 
\end{itemize}
Let $\CM$ denote the graph constructed by the above procedure.
Note that $\CM$ may contain self-loops or multiple edges. 
Given any degree sequence, let $\mathrm{UM}_n(\bld{d})$ denote the graph chosen uniformly at random from the collection of all simple graphs with degree sequence $\boldsymbol{d}$.
It can be shown that the conditional law of  $\mathrm{CM}_{n}(\boldsymbol{d})$, conditioned on it being simple, is the same as $\mathrm{UM}_n(\bld{d})$ (see e.g. \cite[Proposition 7.13]{RGCN1}).

\subsection{Main results}
Fix a constant $\tau\in (3,4)$, which will denote the power-law exponent of the asymptotic degree distribution of $\CM$. 
Throughout this paper we will use the shorthand notation
\begin{equation}\label{eqn:notation-const}
	\alpha= 1/(\tau-1),\quad \rho=(\tau-2)/(\tau-1),\quad \eta=(\tau-3)/(\tau-1).
\end{equation}
We use the standard notation of $\xrightarrow{\sss\PR}$ and $\xrightarrow{\sss d}$ to denote convergence in probability and in distribution, respectively. 
Also, we use a generic notation $C$ to denote a positive universal constant whose exact value may change from line to line.
We use Bachmann–Landau asymptotic notation $o(\cdot)$, $O(\cdot)$, $\Theta (\cdot)$, $\omega (\cdot)$, $\Omega(\cdot)$. 
A sequence of events $(\mathcal{E}_n)_{n\geq 1}$ is said to occur with high probability~(whp) with respect to the probability measures $(\mathbbm{P}_n)_{n\geq 1}$  when $\mathbbm{P}_n\big( \mathcal{E}_n \big) \to 1$. 
For (random) variables $X_n$ and $Y_n$, define $X_n = O_{\sss\mathbbm{P}}(Y_n)$ when  $ ( |X_n|/|Y_n| )_{n \geq 1} $ is a tight sequence; $X_n =o_{\sss\mathbbm{P}}(Y_n)$ when $X_n/Y_n  \xrightarrow{\sss\PR} 0 $; $X_n =\thetaP(Y_n)$ if both $X_n=\OP(Y_n) $ and $Y_n=\OP(X_n)$.

We first state the general assumptions that are used to prove scaling limits for critical configuration models with heavy-tailed degree distributions as identified previously in~\cite{DHLS16,BDHS17}:
\begin{assumption}[General assumptions]\label{assumption1}
	\normalfont 
	For each $n\geq 1$, let $\bld{d}=\boldsymbol{d}_n=(d_1,\dots,d_n)$ be a degree sequence satisfying $d_1\geq d_2\geq\ldots\geq d_n$. 
	We assume the following about $(\boldsymbol{d}_n)_{n\geq 1}$ as $n\to\infty$:
	\begin{enumerate}[(i)] 
		\item \label{assumption1-1} (\emph{High-degree vertices}) For each fixed $i\geq 1$, 
		\begin{equation}\label{defn::degree}
			n^{-\alpha}d_i\to \theta_i,
		\end{equation}
		where $\boldsymbol{\theta}=(\theta_1,\theta_2,\dots) \in \ell^3_{\shortarrow} \setminus \ell^2_{\shortarrow}$, where $\red{\ell^p_{\shortarrow}}:=\{(x_i)_{i\geq 1}: x_1\geq x_2\geq \dots \text{ and }\sum_{i}x_i^p<\infty\}$.  
		\item \label{assumption1-2} (\emph{Moment assumptions}) 
		Let $D_n$ denote the degree of a typical vertex, i.e., a vertex chosen uniformly at random from the vertex set $[n]$, independently of $\mathrm{CM}_n(\boldsymbol{d})$. Then, $D_n$ converges in distribution to some discrete random variable $D$ and 
		\begin{gather}
			\E[D_n] = \frac{1}{n}\sum_{i\in [n]}d_i\to \mu := \E[D],  \qquad \E[D_n^2] = \frac{1}{n}\sum_{i\in [n]}d_i^2 \to \mu_2:=\E[D^2],\label{eqn:669}\\ 
			\lim_{K\to\infty}\limsup_{n\to\infty}n^{-3\alpha} \sum_{i=K+1}^{n} d_i^3=0.\label{eqn:670}
		\end{gather}
		\item  Let $n_1$ be the number of degree-one vertices. Then $n_1=\Theta(n)$, which is equivalent to assuming that $\prob{D=1}>0$.
	\end{enumerate}
\end{assumption}
\begin{remark} \normalfont \label{rem:assumption-1}
	As important examples, Assumption~\ref{assumption1} was shown to hold when the degree distribution is 
	power-law with exponent $\tau\in (3,4)$ \cite[Section 2]{DHLS16}.
	More precisely, 
	if $F$ is a distribution function on the  nonnegative integers  satisfying $[1-F](x) = (1+o(1))C x^{-(\tau-1)}$ as $x\to\infty$, then Assumptions~\ref{assumption1}(i),~\ref{assumption1}(ii) are satisfied when (a)
	$d_i= [1-F]^{-1}(i/n)$,
	and when (b) $d_i$ are the order statistics of an i.i.d.~sample from $F$ (we add a dummy half-edge to vertex 1 if $\sum_{i\in [n]} d_i$ is odd). Assumptions~\ref{assumption1}(iii) is also satisfied in these examples if $F$ has non-zero mass at 1. 
\end{remark}

We further assume that the configuration model lies within the critical window of the phase transition, i.e., for some $\lambda\in \R$,
\begin{equation}\label{defn:criticality}
	\nu_n=\frac{\sum_{i\in [n]}d_i(d_i-1)}{\sum_{i\in [n]}d_i} =  1 + \lambda n^{-\eta} + o(n^{-\eta}).
\end{equation}   
Denote the $i$-th largest connected component of $\rCM_n(\bld{d})$ by $\csi$, breaking ties arbitrarily.
For each $v\in [n]$ and $\delta>0$, let $\mathcal{N}_v(\delta)$ denote the $\delta n^{\eta}$ neighborhood of $v$ in $\rCM_n(\bld{d})$ in the graph distance. 
For each $i\geq 1$, define 
\begin{equation} \label{eq:m-i-defn}
	\mathfrak{m}_i^n(\delta) = \inf_{v\in\mathscr{C}_{\sss (i)}}n^{-\rho} |\mathcal{N}_v(\delta)|.
\end{equation}
Our goal is to prove the global lower mass-bound property for the critical components~$\mathscr{C}_{\sss (i)}$. 
For $\CM$ satisfying Assumption~\ref{assumption1} and \eqref{defn:criticality}, it was shown in \cite[Theorem 1]{DHLS16} that 
\begin{eq}\label{eq:comp-size-conv}
	(n^{-\rho} |\sC_{\sss (i)}|)_{i\geq 1} \dto (\xi_i)_{i\geq 1},
\end{eq}
with respect to the $\ell^2_{\shortarrow}$-topology, where the  $\xi_i$'s are non-degenerate random variables with support $(0,\infty)$. 
Therefore, it is enough to rescale by $n^{\rho}$ in \eqref{eq:m-i-defn} instead of the component sizes as given in \eqref{eq:defn:GLM}. 
In order to prove tightness of $\mathfrak{m}_i^n(\delta)$, we will need a further technical assumption on the degrees. 

\red{
	\begin{assumption}\label{assumption-extra} \normalfont
		Let $V_n^*$ be a vertex chosen in a size-biased manner with sizes being $(d_i/\ell_n)_{i\in [n]}$, i.e., $\PR(V_n^* = i) = d_i/\ell_n$, and 
		let $D_n^*$ be the degree of $V_n^*$.
		There exist constants $c_0>0$ and $c_1 >1$ such that for all $n\geq 1$,
		\begin{eq}\label{eq:defn-D-n-lb}
			\PR(l <D_n^*\leq c_1l) \geq \frac{c_0}{l^{\tau-2}} 
			\ \ \text{ for }\ \
			1\leq l< d_1\, .
		\end{eq}
	\end{assumption}
}

\begin{remark} \label{rem:assumption-2}\normalfont 
	\red{ 
		Assumption~\ref{assumption-extra} says that the mass distribution in the tail of $D_n^*$ is well-behaved in the sense that we have a uniform (over $n$) lower bound of the form \eqref{eq:defn-D-n-lb}.
		Such lower bounds can be used to obtain tail-bounds on the heights of branching processes; see Proposition~\ref{prop:RW-hitting-estimate} below. 
		(See also \cite[Theorem 1.3]{A17}.)
		It can be easily shown that Assumption~\ref{assumption-extra} holds in the examples discussed in Remark~\ref{rem:assumption-1} by observing that the size-biased distribution is a power-law with exponent $\tau -1$. }
\end{remark}

\noindent  The following theorem is the main result of this paper:
\begin{theorem}[Global lower mass-bound for $\CM$]
	\label{thm:gml-bound} 
	Suppose that {\rm Assumptions~\ref{assumption1},~\ref{assumption-extra}} and the criticality condition \eqref{defn:criticality} hold.
	Then, for each  fixed $i\geq 1$, $(\sC_{\sss (i)},n^{-\eta})_{n\geq 1}$ satisfies the  global lower mass-bound, i.e., for any $\delta>0$, the sequence $(\mathfrak{m}_i^n(\delta)^{-1})_{n \geq 1}$ is tight. 
\end{theorem}
\noindent 
By \cite[Theorem 1.1]{J09c}, under the condition \eqref{eqn:669} in Assumption~\ref{assumption1},
\begin{equation}
	\liminf_{n\to\infty} \PR(\CM \text{ is simple})>0.
\end{equation} 
This immediately implies the following:
\begin{theorem}[Global lower mass-bound for $\UM$]\label{cor:GLM-uniform}
	Under {\rm Assumption~\ref{assumption1},~\ref{assumption-extra}} and \eqref{defn:criticality}, the largest components of $\UM$ also satisfy the global lower mass-bound property.
\end{theorem}
Next we state another important corollary, which says that the global lower mass-bound property is also satisfied by critical percolation clusters in $\CM$ and $\UM$. 
To this end, let us assume that 
\begin{equation} \label{eq:defn-super-crit}
	\lim_{n\to\infty}\frac{\sum_{i\in [n]}d_i(d_i-1)}{\sum_{i\in [n]}d_i} = \nu >1.
\end{equation} 
In this regime, $\CM$ is supercritical in the sense that there exists a unique \emph{giant} component whp for $\nu>1$, and when $\nu<1$, all the components have size $\oP(n)$  \cite{JL09,MR95}.
Percolation refers to deleting each edge of a graph independently with probability $1-p$.  
The critical window for percolation on $\CM$ in the heavy-tailed setting was studied in \cite{DHLS16,BDHS17}, and is defined by the values of $p$ given by
\begin{equation}\label{eq:critical-window-defn}
	p_c(\lambda) = \frac{1}{\nu_n}+\frac{\lambda}{n^{\eta}}+o(n^{-\eta}).
\end{equation}
Let $\sC_{\sss (i)}(p_c(\lambda))$ denote the $i$-th largest component of the graph obtained by percolation with probability $p_c(\lambda)$ on the graph $\CM$. 
Then the following result holds:
\begin{theorem}[Global lower mass-bound for critical percolation]\label{cor:GLM-percoltion}
	Under {\rm Assumptions~\ref{assumption1}(i), \ref{assumption1}(ii), \ref{assumption-extra},} \eqref{eq:defn-super-crit} and \eqref{eq:critical-window-defn}, $(\sC_{\sss (i)}(p_c(\lambda)),n^{-\eta})_{n\geq 1}$ satisfies the global lower mass-bound property, for each fixed $i\geq 1$. This result also holds for percolation on $\UM$.
\end{theorem}
Let $\cG_n$ denote the graph obtained by doing percolation with edge retention  probability $p_c(\lambda)$ (defined in \eqref{eq:critical-window-defn}) on $\CM$.
Let $\bld{d}^p=(d_i^p)_{i\in [n]}$ denote the degree sequence of $\cG_n$. 
By \cite[Lemma 3.2]{F07}, the conditional law of $\cG_n$, conditionally on $\bld{d}^p$, is same as the law of $\rCM_n(\bld{d}^p)$. 
Thus, Theorem~\ref{cor:GLM-percoltion} follows from Theorem~\ref{thm:gml-bound} if we can show that the percolated degree sequence~$\bld{d}^p$ satisfies (with possibly different parameters) Assumptions~\ref{assumption1}~and~\ref{assumption-extra} with high probability when the original degree sequence $(d_i)_{i\in [n]}$ satisfies Assumptions~\ref{assumption1}(i),~\ref{assumption1}(ii),~\ref{assumption-extra}, and also  \eqref{defn:criticality} holds \blue{for~$\bld{d}^p$ if further the percolation probability is given by \eqref{eq:critical-window-defn}}. 
The verification of these assumptions are provided in Section~\ref{sec:perc-degrees}. 

\begin{remark}\normalfont 
	It is worthwhile to point out that Theorem~\ref{thm:gml-bound} can be proved when the $\sC_{\sss (i)}$'s are endowed with a more general measure rather than the counting measure. 
	To be precise, for any sequence of vertex weights $(w_v)_{v\in [n]}$, the component $\sC_{\sss (i)}$ can be equipped with the measure $\mu_{\sss (i)} (A) = \sum_{v\in A} w_v / \sum_{v\in \sC_{\sss (i)}} w_v$, for any $A \subset \sC_{\sss (i)}$. 
	Then Theorem~\ref{thm:gml-bound} can also be proved using identical methods as in this paper, with the additional assumptions that 
	$$\lim_{n\to\infty}\frac{1}{\ell_n}\sum_{i\in [n]} d_i w_i = \mu_{w}, \quad \max\bigg\{\sum_{i\in [n]}d_iw_i^2,\sum_{i\in [n]}d_i^2w_i\bigg\} = O(n^{3\alpha}). $$
	These additional assumptions are required when we apply the results from \cite{DHLS16} (see \cite[Theorem 21]{DHLS16}). 
	We adopted the simpler version of the counting measure here because it relates directly to \cite[Theorem 2.1]{BDHS17}. 
\end{remark}

\subsection{Discussion} \label{sec:discussion}
\paragraph*{Scaling limit of critical percolation clusters.}
We write $n^{-\eta} \mathscr{C}_{\sss (i)} (p_c(\lambda))$ to denote the $i$-th largest component of $\mathrm{CM}_n(\bld{d},p_c(\lambda))$, viewed as a measured metric space with the metric being the graph distance re-scaled by $n^{\eta}$, and the measure being proportional to the counting measure.
Athreya, L\"ohr, and Winter~\cite{ALW16-2} showed that the global lower mass-bound property forms a crucial ingredient to prove convergence of random metric spaces such as $n^{-\eta} \mathscr{C}_{\sss (i)} (p_c(\lambda))$ with respect to the Gromov-Hausdorff-Prokhorov (GHP) topology on the space of compact metric spaces. 
The other key ingredient is the scaling limit for $n^{-\eta}\sC_{\sss (i)}(p_c(\lambda))$ with respect to the Gromov-weak topology, which was established in \cite[Theorem 2.1]{BDHS17}.
The Gromov-weak topology is an analogue of  finite-dimensional convergence, since it considers distances between a finite number of sampled points from the underlying metric space.
Thus, global functionals such as the diameter are not continuous with respect to this topology. 
Indeed, it may be the case that there is a long path of growing length, that has asymptotically negligible mass.
In our context, the problem could arise due to paths of length much larger than $ n^{\eta}$. 
The global lower mass-bound property ensures that the components have sufficient mass everywhere. This forbids the existence of long thin paths, when the total mass of the component converges. 
For this reason, Gromov-weak convergence and global lower mass-bound together imply GHP-convergence when the support of the limiting measure is the entire limiting space \cite[Theorem 6.1]{ALW16-2}.
For formal definitions of the Gromov-weak topology, and the GHP-topology on the space of compact measured metric spcaes, we refer the reader to \cite{BHS15,GPW09,ALW16-2}.

Following the above discussion, the next theorem is a direct  consequence of  Theorem~\ref{cor:GLM-percoltion}, \cite[\red{Theorem 2.3}]{BDHS17} and \cite[Theorem 6.1]{ALW16-2}: 
Let $\mathbb{M}$ denote the space of measured compact metric spaces equipped with the GHP-topology, and let $\mathbb{M}^\N$ denote the product space with the associated product topology.
\begin{theorem}[GHP convergence of critical percolation clusters]\label{thm:GHP}
	There exists a sequence of measured metric spaces $(\blue{\mathscr{M}_i})_{i\geq 1} = ((M_i,\dst_i,\mu_i))_{i\geq 1} \in \mathbb{M}^\N$ such that, under {\rm Assumptions~\blue{\ref{assumption1}(i),~\ref{assumption1}(ii),}~\ref{assumption-extra}}\blue{, \eqref{eq:defn-super-crit}} and \eqref{eq:critical-window-defn}, as $n\to\infty$,
	\begin{eq}
		(n^{-\eta} \mathscr{C}_{\sss (i)}(p_c(\lambda)))_{\blue{i\geq 1}} \dto (\blue{\mathscr{M}_i})_{i\geq 1} \quad \text{ in } \ \mathbb{M}^\N.
	\end{eq}
	Moreover, the results also hold for $\mathrm{UM}_n(\bld{d},p_c(\lambda))$.
\end{theorem}

The exact description of the space $\mathscr{M}_i$ can be found in \cite{BDHS17}. 
It is worthwhile mentioning a recent work by Conchon-Kerjan and Goldschmidt~\cite{CG20} which is closely related to Theorem~\ref{thm:GHP}. 
Conchon-Kerjan and Goldschmidt~\cite{CG20} deduce scaling limits for the vector of components in GHP-topology for critical configuration models having i.i.d power law degrees with exponent $\tau\in (3,4)$. 
\blue{In Remarks~\ref{rem:assumption-1} and \ref{rem:assumption-2}, we noted that  Assumptions~\ref{assumption1}(i),~\ref{assumption1}(ii),~and~\ref{assumption-extra} hold when the degrees are i.i.d samples from a power-law distribution with exponent $\tau \in (3,4)$. Therefore, Theorem~\ref{thm:GHP}  implies that the conditional law of $(n^{-\eta} \mathscr{C}_{\sss (i)}(p_c(\lambda)))_{i\geq 1}$, conditioned on the i.i.d degree sequence, converges to the law of $(\blue{\mathscr{M}_i})_{i\geq 1}$  in  $\mathbb{M}^\N$ for almost every realization of the i.i.d degree sequence. Hence,  Theorem~\ref{thm:GHP} gives a quenched result whereas \cite{CG20} proves an annealed result.} The method of \cite{CG20}  relies on an alternative approach showing convergence of the height processes corresponding to the components. 
The associated limiting object was studied in \cite{GHS18},  which interestingly turns out to have a quite different description than those in \cite{BDHS17,BHS15}.

\paragraph*{Scaling limit of maximal distances.}
For any metric space $(X,\dst)$ and a point $x\in X$, define the radius of $x$ in $X$ and the diameter of $X$ by 
\begin{eq}
	\mathrm{Rad}(x,X) = \sup_{y\in X} \dst(x,y)\quad \text{and} \quad \mathrm{diam}(X) = \sup_{x\in X} \mathrm{Rad}(x,X) = \sup_{x,y\in X} \dst(x,y).
\end{eq}
An important corollary of Theorem~\ref{thm:GHP} is the convergence of the radius and the diameter of the critical components:
Let $V_{n,i}$ be a uniformly chosen vertex  in $\mathscr{C}_{\sss (i)}(p_c(\lambda))$, where $(V_{n,i})_{i\geq 1}$ is an independent collection conditionally on $(\mathscr{C}_{\sss (i)}(p_c(\lambda)))_{i\geq 1}$. 
Similarly, using the notation of the scaling limits in Theorem~\ref{thm:GHP},
let $V_i$ be chosen from $M_i$ according to the measure $\mu_i$ and let $(V_i)_{i\geq 1}$ be an independent collection conditionally on $(\blue{\mathscr{M}_i})_{i\geq 1}$.

\begin{corollary}[Convergence of radius and diameter] \label{cor:diameter}
	Under {\rm Assumptions~\blue{\ref{assumption1}(i),~\ref{assumption1}(ii),}~\ref{assumption-extra}}\blue{, \eqref{eq:defn-super-crit}} and \eqref{eq:critical-window-defn}, as $n\to\infty$,
	\begin{eq} 
		\big(n^{-\eta}\mathrm{Rad} (V_{n,i},\mathscr{C}_{\sss (i)}(p_c(\lambda)))\big)_{i\geq 1} &\dto (\mathrm{Rad}(V_i,\blue{\mathscr{M}_i}))_{i\geq 1}, \\
		\big(n^{-\eta}\diam(\sC_{\sss (i)}(p_c(\lambda)))\big)_{i\geq 1} &\dto (\mathrm{diam}(\blue{\mathscr{M}_i}))_{i\geq 1},
	\end{eq}
	with respect to the product topology, where $(\blue{\mathscr{M}_i})_{i\geq 1}$ is given by  {\rm Theorem~\ref{thm:GHP}}. 
	Moreover, the result also holds for $\UM$.
\end{corollary}
Proving scaling limits for the diameter of the critical tree-like objects is often a difficult task.  
In \cite{Sze83}, Szekeres proved that, for the uniform random rooted labelled tree on $m$ vertices, the diameter, rescaled by $\sqrt{m}$, converges in distribution. 
Szekeres also provided an explicit formula for the density of the limiting distribution in \cite[Page 395, (12)]{Sze83}. 
Szekeres' method was based on generating functions. {\L}uczak~\cite{Luc95} also considered enumeration of trees with diameter $\gg \sqrt{m}$. 
On the other hand, Aldous~\cite{Ald91} (see \cite[Section 3.4]{Ald91}) noted that the GHP-convergence can be used as an effective tool to prove scaling limit results for the diameter. 
This is the motivating idea behind Corollary~\ref{cor:diameter}.
Aldous~\cite{Ald91} also raised a natural question whether it is possible to obtain an explicit formula from a result such as Corollary~\ref{cor:diameter}.
In a recent paper, Wang~\cite{Wang15} showed that it is indeed possible to get such a formula for the Brownian tree. In the context of Corollary~\ref{cor:diameter}, the difficulty is two-fold: 
First, the critical components have surplus edges. 
For the scaling limits of critical Erd\H{o}s-R\'enyi random graphs, Miermont and Sen \cite{MS19} recently gave a breadth-first construction, which yields an alternative description of the scaling limit of the radius function from a fixed point (rescaled by $n^{1/3}$).
However, the description for the diameter and an explicit formula such as the one by Wang~\cite{Wang15} is still an open question.
Second, the scaling limit in Corollary~\ref{cor:diameter} is in the heavy-tailed universality class. Even for $\mathbf{p}$-trees (see \cite{CP99}) that satisfy $p_i/(\sum_i p_i^2)^{1/2} \to \beta_i>0$, with $(\beta_i)_{i\geq 1}\in \ell^2_{\shortarrow}\setminus\ell^1_{\shortarrow}$, obtaining an explicit description for the \blue{limiting} distribution of the diameter is an interesting question.

\paragraph*{Compactness of the limiting metric space.}
The \red{limiting spaces $\mathscr{M}_i$ are constructed by tilting the distribution of an inhomogeneous continuum random tree (ICRT), and then identifying a Poisson number of vertices to create cycles. 
	This object is well-defined as a metric measure space for $\bld{\theta} \in \ell^{3}_{\shortarrow} \setminus \ell^2_{\shortarrow}$.
	However, it may not be compact for all $\bld{\theta} \in \ell^{3}_{\shortarrow} \setminus \ell^2_{\shortarrow}$.}
It is interesting to find an explicit criterion  for the compactness of the limiting objects $\mathscr{M}_i$ in terms of the underlying parameters.

Indeed, in the context of compactness of \blue{ICRTs}, Aldous, Miermont, and Pitman \cite[Section~7]{AMP04} \blue{state} an additional condition, which was conjectured to be necessary and sufficient  for the compactness of \blue{ICRTs}.
This conjectured was recently proved in \cite{blanc2022}.
In the context of critical random graphs, a recent paper by Broutin, Duquesne, and Wang \cite{BDW18} shows that
the following criterion, analogous to \cite{AMP04}, is sufficient for the almost sure compactness of $\mathscr{M}_i$ \footnote{\red{\textbf{Note: }The condition \eqref{eq:iff-compactness} does not hold for all $\bld{\theta} \in \ell^3_{\shortarrow} \setminus \ell^2_{\shortarrow}$. Indeed,  take $\theta_i = i^{-1/2} $. For $u\in (\theta_2^{-1},\infty )$, let $i_0 = i_0(u)$ be such that $\theta_{i_0}^{-1}<u\leq \theta_{i_0+1}^{-1}$, i.e., $\sqrt{i_0} < u\leq \sqrt{i_0+1}$. Then, 
		\begin{eq}
			\Psi_{\theta}(u) \leq  C\bigg[\sum_{i\leq i_0}\theta_i (u\theta_i) +\sum_{i>i_0} \theta_i (u\theta_i)^2\bigg] =C\bigg[u\sum_{i< u^2}\frac{1}{i} + u^2\sum_{i\geq u^2} \frac{1}{i^{3/2}}\bigg]  \leq  C[u \log u+u], 
		\end{eq}
		and thus \eqref{eq:iff-compactness} cannot hold.}}:
\begin{eq}\label{eq:iff-compactness}
	\int_1^\infty \frac{\dif u}{\Psi_{\theta} (u)} <\infty, \quad \text{where} \quad \Psi_{\theta} (u) = \sum_{i\geq 1} \theta_i (\e^{-u\theta_i} -1+u\theta_i).
\end{eq}
Our GHP convergence \blue{from Theorem~\ref{thm:GHP}} indirectly yields a  sufficient condition for the compactness of the limiting metric space almost surely, \red{by considering an asymptotic version of Assumption~\ref{assumption-extra}:  
	Suppose $\bld{\theta} \in \ell^3_{\shortarrow} \setminus \ell^2_{\shortarrow}$ and there exist constants $c_0>0$ and $c_1>1$ such that}
\begin{eq}\label{eq:compactness}
	\red{x^{\tau-2} \times \sum_{i=1}^\infty \theta_i \ind{x < \theta_i \leq c_1x} \geq c_0
		\ \  \text{ for all } x\in (0,\theta_1).  }
\end{eq}
\red{The fact that \eqref{eq:compactness} is a sufficient condition for the compactness of $\mathscr{M}_i$ follows immediately from Theorem~\ref{thm:GHP} and the following proposition: 
	\begin{proposition}\label{prop:deg-compact}
		Consider any $\bld{\theta} \in  \ell^3_{\shortarrow} \setminus \ell^2_{\shortarrow}$ such that \eqref{eq:compactness} holds. Then there exists a sequence of degree sequences satisfying {\rm Assumptions~\ref{assumption1}(i), \ref{assumption1}(ii), \ref{assumption-extra}}, and \eqref{eq:defn-super-crit}.
	\end{proposition}
	\noindent We will prove Proposition~\ref{prop:deg-compact} in Appendix~\ref{sec:appendix-comapctness}. A natural question is how the conditions in \eqref{eq:iff-compactness} and \eqref{eq:compactness}  compare. We argue below that, in fact,  \eqref{eq:compactness} is strictly stronger than \eqref{eq:iff-compactness}. }


\red{Recall that $C>0$ is a generic notation for a constant whose value can be different in different expressions. We first show that \eqref{eq:compactness} implies \eqref{eq:iff-compactness}. 
	Suppose $\theta_i >\theta_{i+1}$.
	Then
	\begin{eq}\label{cond:simplified}
		\theta_{i+1}^{\tau-2} \sum_{j=1}^i \theta_j 
		\geq 
		\theta_{i+1}^{\tau-2} \sum_{j=1}^\infty \theta_j\cdot\ind{\theta_{i+1}<\theta_j\leq c_1\theta_{i+1}} \geq c_0,
	\end{eq}
	where the last step uses~\eqref{eq:compactness}.
}
Now, for $u\in (\frac{1}{\theta_i}, \frac{1}{\theta_{i+1}}]$, 
\begin{eq}\label{eq:psi-theta-lb}
	\Psi_{\theta} (u) 
	&\geq 
	C \bigg[\sum_{k=1}^i u\theta_k^2 + \sum_{k=i+1}^\infty u^2 \theta_k^\blue{3}\bigg] 
	\geq 
	Cu\theta_{i+1}\sum_{k=1}^i \theta_k \\
	&\red{= 
		\frac{Cu\theta_{i+1}^{\tau -2 }\sum_{k=1}^i \theta_k}{\theta_{i+1}^{\tau -3 }} }
	\geq 
	\frac{Cu c_0}{\theta_{i+1}^{\tau -3 }} 
	\geq 
	C c_0 u^{\tau-2}.
\end{eq}
Thus, 
$\int_{\theta_1^{-1}}^\infty \frac{\dif u}{\Psi_{\theta} (u)} 
\leq 
C \int_{\theta_1^{-1}}^\infty u^{-(\tau-2)} \dif u 
<
\infty$, 
since $\tau >3$.
This yields~\eqref{eq:iff-compactness}.

\red{To see that the implication is strict, take  $\theta_i = (i^{\alpha} \log (i+2))^{-1} $. 
	Then
	\begin{eq}\label{eqn:1}
		\theta_{i+1}^{\tau -2} \sum_{j=1}^i \theta_j \leq \big((i+1)^\alpha\log (i+3)\big)^{-(\tau-2)} \sum_{j=1}^i j^{-\alpha}\leq \frac{C}{\log ^{\tau-2} i},
	\end{eq}
	which tends to zero as $i\to\infty$. 
	However, as we have seen in \eqref{cond:simplified}, \eqref{eq:compactness} would imply that the left side of~\eqref{eqn:1} is bounded away from zero.
	Thus, \eqref{eq:compactness} does not hold in this case.
	To see that \eqref{eq:iff-compactness} does hold, note that $\theta_{i} \geq \theta_i':= i^{-\alpha'}$ for all large enough $i$, where $\alpha'=\frac{1}{\tau'-1}$ and $3<\tau'<\tau$. 
	Then $(\theta_i')_{i\geq 1}$ satisfies \eqref{eq:compactness}. 
	Therefore, a computation similar to \eqref{eq:psi-theta-lb} yields, 
	for $u\in (\frac{1}{\theta_i},  \frac{1}{\theta_{i+1}}]$, 
	\begin{eq}\label{eq:psi-theta-lb-2}
		\Psi_{\theta} (u) & \geq 
		Cu\theta_{i+1}\sum_{k=1}^i \theta_k \geq Cu\theta_{i+1}'\sum_{k=1}^i \theta_k'\geq \frac{Cu}{(\theta_{i}')^{\tau -3 }}.
	\end{eq}
	Since $u\leq (i+1)^{\alpha} \log (i+3)$, we can choose $\delta>0$ such that 
	$u^{1+\delta} \leq C i^{\alpha'} =  C/\theta_i'$. 
	Therefore, $\Psi_\theta (u) \geq C u^{1+(\tau-3) (1+\delta)}$. 
	Thus, \eqref{eq:iff-compactness} follows.
}

\paragraph*{Proof ideas and technical motivation for this work.} 
The proof of Theorem~\ref{thm:gml-bound} consists of two main steps, that form the key ideas in the argument. 
The first step is to show that the neighborhoods of the high-degree vertices, called \emph{hubs}, have mass  $\thetaP(n^{\rho})$. 
Secondly, any small~$\varepsilon n^{\eta}$ neighborhood \blue{contains a hub with high probability}.
These two facts, summarized in Propositions~\ref{prop:size-nbd bound} and~\ref{prop:diamter-small-comp} below, together ensure that the total mass of any neighborhood of $\sC_{\sss(i)}$ of radius $\varepsilon n^{\eta}$ is bounded away from zero. 
These two facts were proved in \cite{BHS15} in the context of rank-one inhomogeneous random graphs.
However, the proof techniques are completely different here.
The main advantage in \cite{BHS15} was that the breadth-first exploration of components could be dominated by a branching process with \emph{mixed Poisson} progeny distribution that is \emph{independent of $n$}. 
This allows one to use existing literature to estimate the probabilities that a long path exists in the branching process.
However, such a technique is specific to rank-one inhomogeneous random graphs and does not work in the cases where the above stochastic domination does not hold. 
This was one of the technical motivations for this work.
Moreover, the final section contains results about exponential tail-bounds for the number of edges in large critical components (Proposition~\ref{lem:volume-large-deviation}), as well as a coupling of the neighborhood exploration with a branching process with stochastically larger progeny distribution (Section~\ref{sec:BP-approximation}), which are both interesting in their own right.

\paragraph*{Organization of this paper.}
The rest of this paper is organized as follows: In Section~\ref{sec:proof-main-thm}, we state two key propositions, the first involving the total mass of small neighborhoods, and the second involving a bound on the diameter of a \emph{slightly subcritical} $\CM$. 
The proof of Theorem~\ref{thm:gml-bound} is completed in Section~\ref{sec:proof-main-thm}.
In Section~\ref{sec:total-mass-hubs}, we derive the required bounds on the total mass of small neighborhoods.
In Section~\ref{sec:diameter-after-removal}, we obtain bounds on the diameter of the connected components after removing the high-degree vertices. 
\blue{In Section~\ref{sec:perc-degrees}, we prove Assumptions~\ref{assumption1},~\ref{assumption-extra} for the percolated degree sequence, which allows us to conclude Theorem~\ref{cor:GLM-percoltion}.}

\section{Proof of the global lower mass-bound} \label{sec:proof-main-thm}
In this section, we first state the two key ingredients in Propositions~\ref{prop:size-nbd bound}~and~\ref{prop:diamter-small-comp}, and then complete the proof of Theorem~\ref{thm:gml-bound}.
The proofs of Propositions~\ref{prop:size-nbd bound}~and~\ref{prop:diamter-small-comp} are given in the subsequent sections.
The first ingredient shows that hub $i$ has sufficient mass close to it with high probability:
\begin{proposition}
	\label{prop:size-nbd bound}
	Assume that {\rm Assumptions~\ref{assumption1}}~and \eqref{defn:criticality} hold. Recall that $\mathcal{N}_v(\delta)$ denotes the $\delta n^{\eta}$ neighborhood of $v$. 
	For each fixed $i\geq 1$ and $\varepsilon_2>0$, there exists $\delta_{i, \varepsilon_2}>0$ and $n_{i,\varepsilon_2}\geq 1$ such that, for any $\delta\in (0,\delta_{i,\varepsilon_2}]$ and $n\geq n_{i,\varepsilon_2}$, 
	\begin{equation}\label{eq:size-nbd bound}
		\PR\big(|\mathcal{N}_i(\delta)|\leq  \theta_i \delta n^{\rho}\big)\leq \frac{\varepsilon_2}{2^{i+1}}.
	\end{equation}
\end{proposition}
Next, we need some control on the diameter of the graph after removing the hubs. 
Denote by $\mathcal{G}^{\sss >K}_n$ the graph obtained by removing the vertices $[K] = \{1,\dots,K\}$ having the largest degrees and the edges incident to them from $\CM$. 
Note that $\mathcal{G}^{\sss >K}_n$ is a configuration model conditionally on its degree sequence.
Let $\Delta^{\sss >K}$ denote the maximum of the diameters of the connected components of $\mathcal{G}^{\sss >K}_n$. 
The following proposition shows that, for large $K$, $\Delta^{\sss >K}$ is small  with high probability:
\begin{proposition}\label{prop:diamter-small-comp}
	Assume that {\rm Assumptions~\ref{assumption1},~\ref{assumption-extra}} and \eqref{defn:criticality} hold.  Then, for any $\varepsilon_1,\varepsilon_2 > 0$, there exists $K =K(\varepsilon_1,\varepsilon_2)$ and $n_0=n_{0}(\varepsilon_1,\varepsilon_2)$ such that for all $n\geq n_0$,
	\begin{equation}\label{eq:diamter-small-comp}
		\prob{\Delta^{\sss >K}>\varepsilon_1 n^{\eta}}\leq \frac{\varepsilon_2}{4}. 
	\end{equation}
\end{proposition}
\noindent We now prove Theorem~\ref{thm:gml-bound} assuming Propositions~\ref{prop:size-nbd bound} and \ref{prop:diamter-small-comp}:
\begin{proof}[Proof of Theorem~\ref{thm:gml-bound}] 
	Fix $i\geq 1$ and $\varepsilon_1,\varepsilon_2>0$. 
	For a component $\sC \subset \CM$, we write $\Delta(\sC)$ to denote its diameter.
	Let us choose $K$ and $n_0$ so that~\eqref{eq:diamter-small-comp} holds for all $n\geq n_0$. 
	In view of Proposition~\ref{prop:size-nbd bound}, let $\delta_0 = \min\{\varepsilon_1,\delta_{1,\varepsilon_2},\dots,\delta_{K,\varepsilon_2}\}/2$, and $n_0' = \max\{n_0,n_{1,\varepsilon_2},\dots,n_{K,\varepsilon_2}\}$. 
	Thus, for all $n\geq n_0'$, \eqref{eq:size-nbd bound} is satisfied for all $i\in [K]$.
	Define 
	\begin{equation}
		F_1 := \{\Delta^{\sss >K}< \varepsilon_1 n^{\eta}/2\}, \quad F_2 := \{\Delta(\csi )>\varepsilon_1 n^{\eta}/2\}. 
	\end{equation}
	Notice that, on the event $F_1\cap F_2$, it must be the case that one of the vertices in $[K]$ belongs to~$\csi$, and that the union of the neighborhoods of $[K]$ of radius $\lceil\varepsilon_1 n^{\eta}/2\rceil+1 \approx \varepsilon_1 n^{\eta}/2$ covers~$\csi $. 
	Therefore, given any vertex $v\in \csi$, $\cN_v(\varepsilon_1)$ contains at least one of the neighborhoods~$(\cN_j(\varepsilon_1/2))_{j\in [K]}$. 
	This observation yields that
	\begin{equation}
		\inf_{v\in\csi}n^{-\rho} |\mathcal{N}_v(\varepsilon_1)|\geq \min_{j\in [K]} n^{-\rho}| \mathcal{N}_j(\varepsilon_1/2)| \geq \min_{j\in [K]} n^{-\rho}|\mathcal{N}_j(\delta_0)|.
	\end{equation} 
	Thus, for all $n\geq n_0'$,
	\begin{equation}\label{eq:f1-f2}
		\begin{split}
			&\PR\Big(F_1\cap F_2 \cap \Big\{\inf_{v\in\csi}n^{-\rho}| \mathcal{N}_v(\varepsilon_1)|  \leq  \theta_K \delta_0 \Big\}\Big)\\
			&\qquad \leq \sum_{j\in [K]} \PR\big(|\mathcal{N}_j(\delta)|\leq  \theta_j \delta_0 n^{\rho}\big) \leq \sum_{j=1}^K \frac{\varepsilon_2}{2^{j+1}}\leq \frac{\varepsilon_2}{2} ,
		\end{split}
	\end{equation}where the one-but-last step follows from Proposition~\ref{prop:size-nbd bound}.
	Further, on the event $F_2^c$, $|\mathcal{N}_v(\varepsilon_1)| = |\csi|$ for all $v\in \csi$. 
	Moreover, using \eqref{eq:comp-size-conv}, it follows that $n^{-\rho}|\csi|$ converges in distribution to a random variable with  strictly positive support.
	Using the Portmanteau theorem, the above implies that for any $\delta_0'>0$, there exists $\tilde{n}_0 = \tilde{n}_0(\varepsilon_2,\delta_0')$ such that, for all $n\geq \tilde{n}_0$, 
	\begin{equation}
		\PR\big(n^{-\rho}|\csi|\leq \delta_0'\big)\leq \frac{\varepsilon_2}{4}.
	\end{equation}
	Therefore,
	\begin{equation}\label{eq:f1-f2c}
		\PR\bigg( F_2^c\cap \bigg\{\inf_{v\in\csi}n^{-\rho}| \mathcal{N}_v(\varepsilon_1)|  \leq \delta_0' \bigg\}\bigg)\leq \frac{\varepsilon_2}{4}.
	\end{equation} 
	Now, using \eqref{eq:f1-f2} \blue{and} \eqref{eq:f1-f2c}, together with Proposition~\ref{prop:diamter-small-comp}, it follows that, for any $n\geq \max\{n_0',\tilde{n}_0\}$, and $K$ chosen as above,
	\begin{eq}
		\PR\bigg(\inf_{v\in\csi}n^{-\rho}| \mathcal{N}_v(\varepsilon_1)|  \leq \min\{\delta_0',\theta_K\delta_0\} \bigg)\leq \varepsilon_2.
	\end{eq}
	This completes the proof of Theorem~\ref{thm:gml-bound}.
\end{proof}

\section{Lower bound on the total mass of neighborhoods of hubs} 
\label{sec:total-mass-hubs}
In this section, we prove Proposition~\ref{prop:size-nbd bound}.
\begin{proof}[Proof of Proposition~\ref{prop:size-nbd bound}]
	Let us denote the component of $\CM$ containing vertex~$i$  by $\cs(i)$.
	Consider the breadth-first exploration of $\cs(i)$ starting from vertex~$i$, given by
	the following exploration algorithm \cite{DHLS16}:
	\begin{algo}[Exploring the graph]\label{algo-expl}\normalfont  
		The algorithm carries along vertices that can be alive, active, exploring and killed, and half-edges that can be alive, active or killed. 
		We sequentially explore the graph as follows:
		\begin{itemize}
			\item[(S0)] At stage $l=0$, all the vertices and the half-edges are \emph{alive}, and only the half-edges associated to vertex $i$ are \emph{active}. Also, there are no \emph{exploring} vertices except $i$. 
			\item[(S1)] At each stage $l$, \blue{if there is an exploring vertex,} take an active half-edge $e$ of an exploring vertex $v$ and pair it uniformly to another alive half-edge $f$. Kill $e,f$. If $f$ is incident to a vertex $v'$ that has not been discovered before, then declare all the half-edges incident to $v'$ (if any) active, except $f$. 
			If $\mathrm{degree}(v')=1$ (i.e. the only half-edge incident to $v'$ is $f$) then kill~$v'$. Otherwise, declare $v'$ to be active and larger than all other vertices that are alive. After killing $e$, if $v$ does not have another active half-edge, then kill $v$ also. 
			\blue{If there is no exploring vertex at the beginning of stage $l$, we pick the oldest active half-edge, declare the corresponding vertex to be exploring, and then execute the same process as above.}

			\item[(S2)] Repeat (S1) until there is no active half-edges left.
		\end{itemize}
	\end{algo}
	\noindent  Call a vertex \emph{discovered} if it is either active or killed. Let $\mathscr{V}_l$ denote the set of vertices discovered up to time $l$ and $\mathcal{I}_j^n(l):=\ind{j\in\mathscr{V}_l}$.
	Define the exploration process by
	\begin{equation}\label{def:exploration-process}
		S_n(l)= d_i+\sum_{j\neq i} d_j \mathcal{I}_j^n(l)-2l=d_i+\sum_{j\neq i} d_j \left( \mathcal{I}_j^n(l)-\frac{d_j}{\ell_n}l\right)+\bigg( \frac{1}{\ell_n} \sum_{j\neq i}d_j^2-2\bigg)l.
	\end{equation} 
	Note that the exploration process keeps track of the number of active half-edges.
	Thus, $\cs(i)$ is explored when $\bld{S}_n$ hits zero.
	Moreover, since one edge is explored at each step, the hitting time of zero is the total number of edges in $\cs(i)$.
	Define the re-scaled version $\bar{\bld{S}}_n$ of~$\bld{S}_n$ by $\bar{S}_n(t)= n^{-\alpha}S_n(\lfloor tn^{\rho} \rfloor)$. 
	Then, by Assumption~\ref{assumption1} and \eqref{defn:criticality},
	\begin{equation} \label{eqn::scaled_process}
		\bar{S}_n(t) = \theta_i-\frac{\theta_i^2 t}{\mu}+ n^{-\alpha} \sum_{j\neq i}      d_j\left(\mathcal{I}_j^n(tn^\rho)-\frac{d_j}{\ell_n}tn^{\rho} \right)+\lambda t +o(1).
	\end{equation}
	The convergence of this exploration process was considered in  \cite[Theorem 8]{DHLS16} except for the fact that the exploration process started at zero in \cite{DHLS16}. However, using identical arguments to \cite[Theorem 8]{DHLS16}, it can be shown that 
	\begin{eq}\label{eq:dist-conv-S}
		\bar{\bld{S}}_n\xrightarrow{\sss d} \bld{S}_\infty,
	\end{eq}
	with respect to the Skorohod $J_1$-topology, where
	\begin{equation}
		S_\infty(t) = \theta_i -  \frac{\theta_i^2 t}{\mu} +\sum_{j\neq i}\theta_j\Big(\mathcal{I}_j(t)- \frac{\theta_jt}{\mu}\Big)+\lambda t,
	\end{equation}with $\mathcal{I}_j(s):=\ind{\xi_j\leq s }$ and $\xi_j\sim \mathrm{Exponential}(\theta_j/\mu)$ independently of each other. 
	
	Let $h_n(u)$ (respectively $h_\infty(u)$) denote the first hitting time of $\bar{\bld{S}}_n$ (respectively $\bld{S}_\infty$) of~$u$. 
	More precisely, 
	\begin{eq}
		h_n(u):= \inf\Big\{t: \bar{S}_n(t) \leq u \text{ or }  \lim_{t'\nearrow t}\bar{S}_n( t') \leq u\Big\},
	\end{eq} and define $h_\infty(u)$ similarly by replacing $\bar{S}_n(t)$ by $\bar{S}_\infty(t)$.
	Note that, by \cite[Lemma 36]{DHLS16}, the distribution of $h_\infty(u)$ does not have any atoms and therefore, for any $\varepsilon_2>0$, there exists $\beta_{\varepsilon_2,i}>0$ such that
	\[
	\PR\big(h_{\infty}(\theta_i/2)\leq \beta_{\varepsilon_2,i}\big)\leq \frac{\varepsilon_2}{2^{i+1}}.
	\]
	Now we use the following fact:
	\begin{fact}\label{fact}
		Let $(X_n(t))_{t\geq 0}\dto (X(t))_{t\geq 0}$ in Skorohod $J_1$-topology and let $h(X_n)$ (respectively $h(X)$) denote the hitting time to zero of $X_n$ (respectively $X$). 
		Then, $\liminf_{n\to\infty} \PR(h(X_n) > a) \geq \PR(h(X) >a) $, for all $a>0$.
	\end{fact}
	\begin{proof}
		Let $(f_n)_{n\geq 1}$ be such that $h(f_n) \leq a$ for all $n\geq 1$ and $f_n\to f$ in the Skorohod $J_1$-topology as $n\to\infty$.
		Now, $h(f_n) \leq a$ implies that $\inf_{t\in [0,a]} f_n(t) \red{\leq}  0$. 
		Using \cite[Theorem 13.4.1]{W02}, it follows that $\inf_{t\in [0,a]} f(t) \red{\leq} 0$ and thus $h(f) \leq a$.
		Therefore, we have shown that $\{f\colon h(f)\leq a\}$ is a closed set in the Skorohod $J_1$-topology, and therefore  $\{f: h(f)> a\}$ is an open set.
		The proof follows using the Portmanteau theorem \cite[Theorem 2.1~(iv)]{Bil99}. 
	\end{proof}
	\noindent Using \eqref{eq:dist-conv-S} and Fact~\ref{fact}, there exists  $n_{i,\varepsilon_2}\geq 1$ such that, for all $n\geq n_{i,\varepsilon_2}$,
	\begin{equation}
		\PR(h_n(\theta_i/2)\leq \red{ \beta_{\varepsilon_2,i}})\leq \frac{\varepsilon_2}{2^{i}}.
	\end{equation}
	Our first goal is to show that there exists a $\delta_{i,\varepsilon}$ such that for any $\delta\in (0,\delta_{i,\varepsilon_2}]$,  
	\begin{eq}\label{degree-implies-height}
		\sum_{k\in \cN_i(\delta)}d_k \leq \theta_i\delta n^{\rho}\quad \implies \quad h_n(\theta_i/2)\leq  \red{\beta_{\varepsilon_2,i}}.
	\end{eq}
	Recall that $\mathcal{N}_v(\delta)$ denotes the $\delta n^{\eta}$ neighborhood of $v$ in $\rCM_n(\bld{d})$.
	To prove \eqref{degree-implies-height}, let $\partial(j)$ denote the set of vertices at distance $j$ from $i$. 
	Let $E_{j1}$ denote the total number of edges 
	between  vertices in $\partial(j)$ and $\partial(j-1)$, and let $E_{j2}$ denote the number of edges within the vertices in $\partial(j-1)$.
	Define $E_j = E_{j1}+E_{j2}$.
	Fix any $\delta<2\beta_{\varepsilon_2,i}/\theta_i$. 
	Note that if $\sum_{k\in\mathcal{N}_i(\delta)}d_k\leq \theta_i \delta n^{\rho}$, then the total number of edges in $\mathcal{N}_i(\delta)$ is at most $\theta_i \delta n^{\rho}/2$. 
	Thus there exists $j\leq \delta n^{\eta}$ such that $E_j\leq\theta_i\delta n^{\rho}/2\delta n^{\eta} = \theta_in^{\alpha}/2 $. 
	This implies that $\bld{S}_n$ must go below $\theta_i n^\alpha/2$ before exploring all the vertices in $\cN_i(\delta)$.
	This is because we are exploring the components in a breadth-first manner and $\bar{\bld{S}}_n$ keeps track of the number of active half-edges, which in turn are the potential connections to vertices at the next level.
	Since one edge is explored in each time step, and we rescale time by $n^{\rho}$,  this implies that 
	\begin{equation}
		h_n(\theta_i/2)\leq \frac{1}{2} n^{-\rho} \sum_{k\in\mathcal{N}_i(\delta)}d_k\leq \theta_i\delta/2 \leq \beta_{\varepsilon_2,i}.
	\end{equation}
	Therefore, for all $n\geq n_{i,\varepsilon_2}$,
	\begin{equation} \label{eq:dk-bound}
		\PR\bigg(\sum_{k\in\mathcal{N}_i(\delta)}d_k\leq \theta_i\delta n^{\rho}\bigg)\leq \PR(h_n(\theta_i/2)\leq \beta_{\varepsilon_2,i})\leq \frac{\varepsilon_2}{2^{i}}.
	\end{equation}
	Finally, to conclude Proposition~\ref{prop:size-nbd bound} from \eqref{eq:dk-bound}, we use the \blue{result from \cite[Lemma 22]{DHLS16} that, for any $T>0$,}
	\begin{equation}\label{weight-expl-prop}
		\sup_{u\leq T}\bigg| \sum_{i\in [n]}\mathcal{I}_i^n(un^{\rho})-un^{\rho}\bigg|=\oP(n^{\rho}).
	\end{equation} 
	This implies that the difference between the number of edges and the number of vertices explored up to time $un^{\rho}$ is $\oP(n^{\rho})$ uniformly over $u\leq T$.
	The proof of Proposition~\ref{prop:size-nbd bound} now follows.
\end{proof}

\section{Diameter after removing hubs}
\label{sec:diameter-after-removal}
Throughout the remainder of the paper, we fix the convention that $C,C',C''>0$ etc.~denote constants whose value can change from line to line.
Recall the definition of the graph $\cG_n^{\sss >K}$ from Proposition~\ref{prop:diamter-small-comp}. 
If we keep on exploring $\cG_n^{\sss >K}$ in a breadth-first manner using Algorithm~\ref{algo-expl} and ignore the cycles created, then we get a random tree. 
The idea is to couple neighborhoods in $\cG_n^{\sss >K}$ with a suitable branching process such that the progeny distribution of the branching process dominates the number of children of each vertex in the breadth-first tree.
Therefore, when there is a long path in $\cG_n^{\sss >K}$ that makes the diameter large, that long path must be present in the branching process as well under the above coupling.
In this way, the question about the diameter of $\cG_n^{\sss >K}$ reduces to the question about the height of a branching process. 
To estimate the height suitably, we use a recent beautiful proof technique by Addario-Berry~\cite{A17} which allows one to relate the height of a branching process to the sum of inverses of the associated breadth-first random walk.

In Section~\ref{sec:asymp-edges}, we establish \blue{tail bounds} for the number of edges within components.
This allows us to formulate the desired coupling in Section~\ref{sec:BP-approximation}. 
In Section~\ref{sec:height-vs-rw}, we analyze the breadth-first random walk to show that it is unlikely that the height of the branching process is larger than $\varepsilon n^{\eta}$. 
These bounds are different from those derived in~\cite{A17} since our branching process depends on $n$ and there is a joint scaling involved between the distances and the law of the branching process.

\subsection{Asymptotics for the number of edges} \label{sec:asymp-edges}
For a graph $G$, let $\rE(G)$ denote the number of edges in $G$.
\begin{proposition}\label{lem:volume-large-deviation} Suppose that {\rm Assumption~\ref{assumption1}} and \eqref{defn:criticality} hold. 
	\blue{For all $\varepsilon \in (0,\frac{4-\tau}{\tau-1})$}, and  sufficiently large~$n$,
	\begin{eq}
		\PR(\rE(\sC (i))> n^{\rho+\varepsilon}) \leq C\e^{-C' n^{\varepsilon/2}},
	\end{eq}for some absolute constants $C,C'>0$ and all $i\in [n]$.  
\end{proposition}
The proof of Proposition~\ref{lem:volume-large-deviation} relies on concentration techniques for martingales. 
We start by defining the relevant notation.
Consider exploring $\CM$ with Algorithm~\ref{algo-expl}, and let the associated exploration process be defined in \eqref{def:exploration-process}. 
Let us denote  the degree of the vertex found at step $l$ by $d_{\sss (l)}$. 
If no new vertex is found at step $l$, then $d_{\sss (l)} = 0$.
Also, let $\mathscr{F}_l$ denote the sigma-algebra containing all the information revealed by the exploration process up to time $l$.
Thus, 
\begin{eq}
	S_n(0) = d_i, \quad \text{and}\quad  S_n(l) = S_n(l-1) + (d_{\sss (l)}-2).
\end{eq}
Using the Doob-Meyer decomposition, one can write 
\begin{equation}
	S_n(l) = S_n(0)+M_n(l) + A_n(l), 
\end{equation}where $M_n$ is a martingale with respect to $(\mathscr{F}_l)_{l\geq 1}$. 
The drift $A_n$ and the quadratic variation $\langle M_n \rangle$ of $M_n$ are given by 
\begin{equation}
	A_{n}(l)= \sum_{j=1}^{l} \mathbbm{E}\big[d_{\sss(j)}-2 \vert \mathscr{F}_{j-1} \big], \quad 
	\langle M_n \rangle(l)= \sum_{j=1}^{l} \var{d_{\sss(j)}\vert \mathscr{F}_{j-1}} .
\end{equation}
We will show that for any $\varepsilon \in (0,\varepsilon_0)$, the following two lemmas hold:
\begin{lemma}\label{lem:small-martingale}
	Suppose that {\rm Assumption~\ref{assumption1}} and \eqref{defn:criticality} hold.  \blue{For all $\varepsilon \in (0,\frac{4-\tau}{\tau-1})$},  and  sufficiently large~$n$, 
	\begin{eq}
		\PR(n^{-(\alpha+\varepsilon)} M_n(n^{\rho +\varepsilon}) > 1) \leq C\e^{-C' n^{\varepsilon}},
	\end{eq}
	for some absolute constants $C,C'>0$.
\end{lemma}
\begin{lemma}\label{lem:drift-superlinear}
	Suppose that {\rm Assumption~\ref{assumption1}} and \eqref{defn:criticality} hold.
	For all fixed $K\geq 1$, \blue{$\varepsilon \in (0,\frac{4-\tau}{\tau-1})$}, and sufficiently large~$n$,
	\begin{eq}
		\PR\bigg(n^{-(\alpha+\varepsilon)} A_n(n^{\rho +\varepsilon}) \geq -C\sum_{i=1}^K\theta_i^2\bigg) \leq C\e^{-C' n^{\varepsilon/2}},
	\end{eq}for some absolute constants $C,C'>0$.
\end{lemma}  
\begin{proof}[Proof of Proposition~\ref{lem:volume-large-deviation} subject to Lemmas~\ref{lem:small-martingale},~\ref{lem:drift-superlinear}] 
	Throughout, we write $t_n := n^{\rho +\varepsilon}$.
	Note that we can choose $K\geq 1$ such that $\sum_{i=1}^K\theta_i^2$ is arbitrarily large since  $\bld{\theta}\notin \ell^2_{\shortarrow}$.
	Thus, if $n^{-(\alpha+\varepsilon)} M_n(t_n) \leq 1$ and $n^{-(\alpha+\varepsilon)}A_n(t_n) \leq -C\sum_{i=1}^K\theta_i^2$, then $n^{-(\alpha + \varepsilon)}S_n(t_n) <0$, and therefore $\sC(i)$ must be explored before time $t_n$, and thus  $\rE(\sC(i))\leq t_n$.
	As a result, Lemmas~\ref{lem:small-martingale} and~\ref{lem:drift-superlinear} together complete the proof of Proposition~\ref{lem:volume-large-deviation}.
\end{proof}
\begin{proof}[Proof of Lemma~\ref{lem:small-martingale}]
	First note that $\blue{\frac{4-\tau}{\tau-1} +\rho = 2\alpha} < 1$ and therefore $t_n = o(n)$.
	Thus, uniformly over $j\leq t_n$, 
	\begin{equation}
		\var{d_{\sss (j)} \vert \mathscr{F}_{j-1}} \leq \E[d_{\sss (j)}^2 \vert \mathscr{F}_{j-1}]  = \frac{\sum_{j\notin \mathscr{V}_{j-1}} d_j^3}{\ell_n - 2j +2} \leq \frac{\sum_{j\in [n]} d_j^3}{\ell_n - 2t_n+2}\leq Cn^{3\alpha - 1},
	\end{equation}so that, almost surely,
	\begin{equation}\label{eq:bound-QV}
		\langle M_n\rangle (t_n) \leq t_n Cn^{3\alpha-1} = C n^{2\alpha + \varepsilon}.
	\end{equation}
	Also, $d_{\sss (j)} \leq C n^{\alpha}$ almost surely. We can now use Freedman's inequality \cite[Proposition 2.1]{Fre75} which says that if $Y(k) = \sum_{j\leq k} X_j$ with $\E[X_j\vert \mathcal{F}_{j-1}] =0$ (for some filtration $(\mathcal{F}_j)_{j\geq 1}$) and $\PR(|X_j|\leq R, \ \forall j\geq 1)=1$, then, for any $a,b >0$,
	\begin{eq}\label{eq:freedman}
		\PR(Y(k) \geq a, \text{ and } \langle Y\rangle(k)  \leq b ) \leq \exp\bigg(\frac{-a^2}{2(Ra+b)}\bigg).
	\end{eq}
	We apply \eqref{eq:freedman} with $a= n^{\alpha+\varepsilon}$, $b=Cn^{2\alpha+\varepsilon}$ and $R=Cn^{\alpha}$. 
	Note that $\langle M_n \rangle (t_n) \leq b$ \blue{almost surely} by \eqref{eq:bound-QV}.
	It follows that 
	\begin{eq}
		\PR(M_n(t_n) > n^{\alpha + \varepsilon}) \leq \exp \bigg(- \frac{n^{2\alpha + 2\varepsilon}}{2 \blue{C} (n^{\alpha} n^{\alpha+ \varepsilon} + n^{2\alpha+ \varepsilon})}\bigg) \leq C\e^{-C'n^{\varepsilon}},
	\end{eq} 
	and the proof follows.
\end{proof}

\begin{proof}[Proof of Lemma~\ref{lem:drift-superlinear}]
	Note that 
	\begin{eq}\label{the-split-up}
		&\mathbbm{E} \big[ d_{\sss(i)} -2 \vert \mathscr{F}_{i-1} \big]  = \frac{\sum_{j \notin \mathscr{V}_{i-1}} d_{j}^2}{\ell_n-2i+1}-2\\
		&\hspace{2cm}= \frac{1}{\ell_n}\sum_{j \in [n]} d_{j}(d_{j}-2)- \frac{1}{\ell_n} \sum_{j \in \mathscr{V}_{i-1}} d_{j}^2 + \frac{(2i-1)\sum_{j \notin \mathscr{V}_{i-1}} d_{j}^2 }{\ell_n (\ell_n-2i+1)} \\
		&\hspace{2cm}\leq  \lambda n^{-\eta} - \frac{1}{\ell_n} \sum_{j \in \mathscr{V}_{i-1}} d_{j}^2 + \frac{(2i-1)}{(\ell_n-2i+1)^2} \sum_{j \in [n]} d_{j}^{2} +o(n^{-\eta})
	\end{eq}uniformly over $i\leq t_n$.
	Therefore, for all sufficiently large $n$, 
	\begin{eq}\label{upper-bound-drift-large}
		A_n(t_n)  = \sum_{j=1}^{t_n} \mathbbm{E}\big[d_{\sss(j)}-2 \vert \mathscr{F}_{j-1} \big]&\leq \lambda t_n n^{-\eta} - \frac{1}{\ell_n}\sum_{i=1}^{t_n} \sum_{j \in \mathscr{V}_{i-1}} d_{j}^{2} + \frac{Ct_n^2}{\ell_n}+ o(n^{\alpha+\varepsilon}) \\
		&= \lambda n^{\alpha+\varepsilon} - \frac{1}{\ell_n}\sum_{i=1}^{t_n} \sum_{j \in \mathscr{V}_{i-1}} d_{j}^{2} + o(n^{\alpha+\varepsilon}),
	\end{eq}where in the second step we have used $\sum_{i\in [n]} d_i^2 / \ell_n = O(1)$, and in the last step we have used that $t_n^2/\ell_n = O(n^{2\rho+2\varepsilon-1}) = o(n^{\alpha+\varepsilon})$ for $\varepsilon<  1+ \alpha - 2\rho= \frac{4-\tau}{\tau-1} $.
	Let us denote the second term in \eqref{upper-bound-drift-large} by (A).
	To analyze~(A), define the event 
	\begin{eq}
		\cA_n:= \big\{\exists j: d_j>n^{\alpha - \varepsilon/2}, j\notin \mathscr{V}_{t_n/2}\big\}.
	\end{eq}
	Then, for all sufficiently large $n$,
	\begin{eq}\label{eq:estimate-bad-event-exponential}
		\PR(\cA_n) \leq \sum_{j: d_j>n^{\alpha-\varepsilon/2}} \bigg(1-\frac{d_j}{\ell_n-2t_n}\bigg)^{t_n} \leq n \e^{-Cn^{\varepsilon/2}}.
	\end{eq}
	On the event $\cA_n^c$, 
	\begin{eq}\label{eq:expression-A}
		\mathrm{(A)} = \frac{1}{\ell_n} \sum_{i=1}^{t_n} \sum_{j\in [n]}d_j^2\mathbbm{1}\{j\in \mathscr{V}_{i-1}\} \geq \frac{1}{\ell_n} \sum_{i = \frac{t_n}{2}+1}^{t_n} \sum_{j=1}^K d_j^2  \geq Cn^{\alpha+\varepsilon} \sum_{j=1}^K \theta_j^2.  
	\end{eq}
	Combining \eqref{upper-bound-drift-large}, \eqref{eq:estimate-bad-event-exponential} and \eqref{eq:expression-A}  now completes the proof.
\end{proof}

\subsection{Coupling with branching processes} \label{sec:BP-approximation}
Recall that $\mathscr{C}(i)$ is the connected component in $\CM$ containing vertex $i$. Define the event $\cK_n:= \{\rE(\cs(i))>n^{\rho+\varepsilon}\}$. 
Proposition~\ref{lem:volume-large-deviation} implies that the probability of $\cK_n$ happening is exponentially small in $n$. 
On the event~$\mathcal{K}_n^c$, we can couple the breadth-first exploration starting from vertex $i$ with a suitable branching process. 
Let $n_k$ denote the number of vertices of degree $k$ and consider the branching process $\mathcal{X}_n(i)$ starting with $d_i$ individuals, and the progeny distribution $\bar{\xi}_n$ given by 
\begin{equation}\label{upperbounding-BP}
	\begin{split}\prob{\bar{\xi}_n=k}=\bar{p}_k=
		\begin{cases} 
			\frac{(k+1)n_{k+1}}{\ubar{\ell}_n} \quad &\text{for } k\geq 1,\\
			\frac{n_1-2n^{\rho+\varepsilon}}{\ubar{\ell}_n} \quad &\text{for } k=0,
		\end{cases}
	\end{split}
\end{equation}where $\ubar{\ell}_n=\ell_n-2n^{\rho+\varepsilon}$.
Note that, at each step of the exploration, we have at most $(k+1)n_{k+1}$ half-edges that are incident to vertices having $k$ further unpaired half-edges. 
Further, on the event $\mathcal{K}_n^c$, we have at least $\ubar{\ell}_n$ choices for pairing. 
Therefore, the number of active half-edges discovered at each step in the breadth-first exploration of the neighborhoods of $i$ is stochastically dominated by  $\bar{\xi}_n$.
This proves the next proposition, which we state after setting up some further notation. 
Recall that $\cG_n^{\sss >i-1}$ denotes the graph obtained by deleting vertices in $[i-1]$ and the associated edges from $\CM$.
Let $\partial_i(r)$ denote the number of vertices at distance $r$ from vertex $i$ in the graph $\cG_n^{\sss >i-1}$. 
Let $\bar{\xi}_n(i)$ denote the random variable with the distribution in \eqref{upperbounding-BP} truncated in such a way that $\{d_1,\dots,d_{i-1}\}$ are excluded from the support.
More precisely,
\begin{eq}\label{eqn:668}
	\PR(\bar{\xi}_n(i)=k) =  
	\begin{cases}
		0\quad &\text{for }  k> d_{i}, \\
		\frac{(k+1)}{L}\#\{j\geq i:\ d_j=k+1\} \quad &\text{for } 1\leq k\leq d_{i},\\
		\frac{n_1-2n^{\rho+\varepsilon}}{L} \quad &\text{for } k=0,
	\end{cases}
\end{eq}where $L=\underline{\ell}_n-\sum_{j=1}^{i-1}d_j$ is the appropriate normalizing constant.
Let $\cX_{n,\mathrm{res}}(i)$ denote the branching process starting with~$d_i$ individuals and progeny distribution $\bar{\xi}_n(i)$ and let $\bar{\partial}_i(r)$ denote the number of individuals in generation $r$ of $\mathcal{X}_{n,\mathrm{res}}(i)$.
Then the above stochastic domination argument immediately yields the next proposition:

\begin{proposition}\label{prop:coupling-uppperbound}
	Suppose that {\rm Assumption~\ref{assumption1}} and \eqref{defn:criticality} hold. 
	Let $K_n$ be as described \blue{in {\rm Lemma~\ref{lem:technical}} below}.
	For all $r\geq 1$, $1\leq i\leq K_n$, \blue{$\varepsilon \in (0,\frac{4-\tau}{\tau-1})$} and $n\geq 1$, 
	\begin{eq}
		\PR(\partial_i(r)\neq \varnothing) \leq \PR(\bar{\partial}_i(r)\neq \varnothing)+ \PR(\rE(\sC(i))>n^{\rho+\varepsilon}).
	\end{eq}
\end{proposition}
\noindent Before proceeding with the next section in which we investigate $\PR(\bar{\partial}_i(r)\neq \varnothing)$, we estimate the expectation and variance of the progeny distribution in the branching process $\mathcal{X}_{\sss n,\mathrm{res}}(i)$ using Assumptions~\ref{assumption1},~\ref{assumption-extra}, and \eqref{defn:criticality}. 
Using $\sum_i \theta_i^2 = \infty$, we can choose $i_2(\lambda)$ (depending only on $\lambda$) such that 
\begin{align}\label{eqn:666}
	\frac{1}{\mu}\sum_{i=1}^{i_2(\lambda)}\theta_i^2\geq 5\lambda.
\end{align}
Also the normalizing constant in \eqref{eqn:668} satisfies 
\blue{
	\begin{eq}\label{eq:L-asymp}
		L = \ell_n(1+o(n^{-\eta}))
	\end{eq}
	uniformly over $1\leq i\leq K_n$.
	To see this, first observe that  $\ubar{\ell}_n=\ell_n-2n^{\rho+\varepsilon} = o(n^{-\eta})$ since $\varepsilon < 1-\rho -\eta = \frac{4-\tau}{\tau-1}$. 
	Also, $\frac{1}{\ell_n}\sum_{j\leq i} d_j = O(d_1K_nn^{-1})= o(n^{2\alpha-1})=o(n^{-\eta})$, \blue{as $K_n = o(n^{\alpha})$} by Assumption~\ref{assumption-extra} \blue{and Lemma~\ref{lem:technical}} and $2\alpha -1 = -\eta$.
	Hence, \eqref{eq:L-asymp} follows.}
Now using Assumption~\ref{assumption1} \blue{and \eqref{defn:criticality}}, note that there exists $N_\lambda\geq 1$  such that for all $n\geq N_\lambda$ and $i_2(\lambda)\leq i\leq K_n$,
\begin{eq} \label{eq:computation-mean-BP}
	\expt{\bar{\xi}_n(i)}&=\frac{1}{L}\sum_{j\geq  i} d_j(d_j-1)= \frac{1}{\ell_n}\sum_{j \geq  i} d_j(d_j-1) +o(n^{-\eta})\\
	&=1+\lambda n^{-\eta}-\frac{1}{\ell_n}\sum_{j<i}d_j(d_j-1)+o(n^{-\eta})\\
	&
	\leq
	1+\lambda n^{-\eta}-\frac{1}{2\ell_n}\sum_{j<i}d_j^2+o(n^{-\eta})\\
	&\leq 1 - \bigg(Cn^{-2\alpha}\sum_{j< i} d_j^2\bigg)n^{-\eta}  +o(n^{-\eta }),
\end{eq}
where the third step uses \eqref{defn:criticality}, the penultimate step uses the fact that $d_i\geq 2$ so that $\sum_{j<i} d_j \leq \sum_{j<i}d_j^2 /2$ for $i\leq K_n$, and the last step uses \eqref{eqn:666}.
Thus, for $n\geq N_\lambda$ and $i_2(\lambda)\leq i\leq K_n$,
\begin{eq} \label{eq:BP-expt}
	\expt{\bar{\xi}_n(i)}\leq 1 - \beta_i^n n^{-\eta} \quad \text{where}\quad \beta_i^n=Cn^{-2\alpha}\sum_{j< i}d_j^2.
\end{eq} 
The estimate in \eqref{eq:BP-expt}  will be crucial in the next section.

\subsection{Estimating heights of trees via random walks}
\label{sec:height-vs-rw}
We will prove the following theorem in this section:
\begin{theorem}\label{lem:boundary-small-prob} 
	Suppose that {\rm Assumptions~\ref{assumption1}, \ref{assumption-extra}}, and \eqref{defn:criticality} hold. Fix $\varepsilon>0$. 
	Then, for all $i_2(\lambda)\leq i \leq K_n$ (where $i_2(\lambda)$ \blue{and $K_n$ are} given by \eqref{eqn:666} \blue{and {\rm Lemma~\ref{lem:technical}} respectively}) and $n\geq N_\lambda$,
	\begin{equation}
		\PR(\bar{\partial}_i( \varepsilon n^\eta)\neq \varnothing)\leq \frac{Cd_i}{n^{\alpha}} 
		\e^{- \frac{\varepsilon \beta_i^n}{2}},
	\end{equation}for some constant $C = C(\varepsilon,\lambda)>0$.
\end{theorem}
Define $\cX_n^1 (i)$ to be the Galton-Watson tree starting with one offspring and progeny distribution $\bar{\xi}_n(i)$ and let $\bar{\partial}_i^1(r)$ denote the number of individuals in generation $r$ of $\mathcal{X}_{n}^1(i)$.
The crucial ingredient for the proof of Theorem~\ref{lem:boundary-small-prob} is the following:
\begin{proposition}\label{prop-height-bound-one-progeny}
	Under identical conditions as in {\rm Theorem~\ref{lem:boundary-small-prob}}, for all $n\geq N_\lambda$,
	\begin{eq}
		\PR(\bar{\partial}_{i_2(\lambda)}^1(\varepsilon n^{\eta}) \neq \varnothing) \leq \frac{C}{n^{\alpha}},
	\end{eq}for some constant $C = C(\varepsilon,\lambda)>0$.
\end{proposition}
\begin{proof}[Proof of Theorem~\ref{lem:boundary-small-prob} using Proposition~\ref{prop-height-bound-one-progeny}]
	Let $M_r$ denote the number of children at generation $r$ of $\cX_{n,\mathrm{res}}(i)$,
	and note that
	\begin{eq}
		\PR(\bar{\partial}_i( \varepsilon n^\eta)\neq \varnothing) \leq \E[M_{\varepsilon n^{\eta}/2}] \times  \PR(\bar{\partial}_i^1(\varepsilon n^{\eta}/2) \neq \varnothing).
	\end{eq}
	Now, using \eqref{eq:BP-expt},
	\begin{eq}
		\E[M_{\varepsilon n^{\eta}/2}] \leq  d_i (1-\beta_i^n n^{-\eta})^{\varepsilon n^{\eta}/2} \leq d_i \e^{- \frac{\varepsilon \beta_i^n}{2}},
	\end{eq}
	and $\bar{\xi}_n(i)\preceq \bar{\xi}_n(i-1) \preceq \dots \preceq \bar{\xi}_n(i_2(\lambda))$, where $\preceq$ denotes stochastic domination.
	Thus, 
	\begin{eq}
		\PR(\bar{\partial}_i( \varepsilon n^\eta)\neq \varnothing) \leq d_i \e^{- \frac{\varepsilon \beta_i^n}{2}}\times \PR(\bar{\partial}_{i_2(\lambda)}^1(\varepsilon n^{\eta}/2) \neq \varnothing),
	\end{eq}
	and the proof of Theorem~\ref{lem:boundary-small-prob} follows using Proposition~\ref{prop-height-bound-one-progeny}.
\end{proof}
The rest of this section is devoted to the proof of Proposition~\ref{prop-height-bound-one-progeny}.
We leverage some key ideas from~\cite{A17}.
Define the breadth-first random walk $\bld{s}_n$ by $s_n(0) = 1$ and 
\begin{equation}\label{eq:random-walk-tree}
	s_n(u) = s_n(u-1) + \zeta_u -1, 
\end{equation}
where $(\zeta_u)_{u\geq 1}$ are i.i.d.~observations from the distribution of $\bar{\xi}_n(i_2(\lambda))$. 
Let $\sigma := \inf\{u:s_n(u) = 0\}$ and for $t=0, 1,\ldots, \sigma$, define the function 
\begin{equation}
	H_n(t) := \red{\sum_{u=0}^{t-1}}\frac{1}{s_n(u)}\, .
\end{equation}
A remarkable fact observed in \cite[Proposition 1.7]{A17} states that the height of a tree with breadth-first exploration process $\bld{s}_n$ is at most $3H_n(\sigma)$. 
Thus Proposition~\ref{prop-height-bound-one-progeny} can be concluded directly from the following estimate:
\begin{proposition}\label{prop:RW-hitting-estimate}
	Under identical conditions as in {\rm Theorem~\ref{lem:boundary-small-prob}}, for all $n\geq N_\lambda$,
	\begin{eq}\label{eqn:2}
		\PR(H_n(\sigma) > \varepsilon n^{\eta}) \leq \frac{C}{n^{\alpha}},
	\end{eq}
	for some constant $C = C(\varepsilon,\lambda)>0$.
\end{proposition}

In what follows, we fix $\delta>0$ such that $\delta n^{\alpha} +2 < d_{i_2(\lambda)}/100$ for all $n\geq N_{\lambda}$.
Define $I_l:=[2^{l-1} ,2^{l+1} )$ for $l\geq 1$.
Let $\PR_x$ denote the law of the random walk $\bld{s}_n$, starting from $x$ and satisfying the recurrence relation in \eqref{eq:random-walk-tree}. 
Let $\sigma_{nl} : = \min\{t\geq 1: s_n(t) \notin I_l\}$ and
$r_{nl}: = \min\{t\geq 1:\sup_{x\in I_l}\PR_x(\sigma_{nl}>t)\leq 1/2\}$.  
We first obtain the following bound on $r_{nl}$:
\begin{lemma}\label{lem:r-nl-ub}
	Under identical conditions as in {\rm Theorem~\ref{lem:boundary-small-prob}}, there exists $n_{\star}\geq 1$ depending only on $(d_i\, ;\, i\in [n], n\geq 1)$ such that for all $n\geq n_{\star}$ and all $l\geq 1$ satisfying $2^{l+1}\leq\delta n^{\alpha}$, we have 
	$r_{nl} \leq C 2^{(\tau-2)l}$ for some (sufficiently large) constant $C>0$.
\end{lemma}
\begin{proof}
	\blue{By \eqref{eqn:668},  $\PR(\bar{\xi}_n(i_2(\lambda))= j ) = (1+o(1)) \PR(D_n^* = j+1)$ uniformly over $1\leq j\leq d_{i_2(\lambda)}$. 
		Thus, by Assumption~\ref{assumption-extra}, 
		\begin{eq}\label{eq:lower-bound-tail-xi}
			\PR\Big(\frac{u}{c_1}<\bar{\xi}_n(i_2(\lambda)) \leq u \Big) \geq Cu^{- (\tau-2)},
		\end{eq}for all $c_1\leq u\leq \delta n^\alpha$.} 
	Next, in order to estimate $\sigma_{nl}$, we bound $\sup_{x\in I_l} \PR_x (s_n(t) \in I_l) $ using an upper bound on L\'evy's concentration function due to Esseen~\cite{Ess86}, that we describe now.
	For a random variable $Z$, define L\'evy's concentration function 
	\begin{eq}
		Q(Z,L) := \sup_{x\in \R} \PR(Z\in [x,x+L)).
	\end{eq}
	By \cite[Theorem~3.1]{Ess86}, for any $u>0$,
	\begin{eq}\label{eq:esseen-inequality}
		Q(s_n(t), u ) \leq \frac{Cu}{\big(t \times  \E[|\zeta_1-\zeta_2|^2 \ind{|\zeta_1-\zeta_2| \leq u}]\big)^{1/2}}\, ,
	\end{eq}
	where $\zeta_1$ and $\zeta_2$ are i.i.d.~realizations from the distribution of $\bar{\xi}_n(i_2(\lambda))$.
	\red{To get an upper bound on the right side of} \eqref{eq:esseen-inequality}, we first observe that for any random variable $Y $ supported on $\bZ_{\geq 0}$, 
	\begin{eq}
		\E[Y^2 \1\{Y\leq u\}] &= \sum_{1\leq y\leq u} y^2 \PR(Y = y) =  \sum_{1\leq y\leq u} \sum_{1\leq x\leq y}y \PR(Y = y) \\
		&= \sum_{1\leq x\leq u} \sum_{ \blue{x\leq y\leq u}} y \PR(Y = y) \geq \sum_{1\leq x\leq u} x \PR(\blue{x\leq Y\leq u}).
	\end{eq}

	Now, it follows from \eqref{eqn:668} and Assumption~\ref{assumption1}~(iii) that $\liminf_{n\to\infty}\PR(\bar{\xi}_n(i_2(\lambda)) =0)>\nobreak0$.
	Similarly, using \eqref{eqn:668} and Assumption~\ref{assumption1}~(ii), we can choose an integer $j_{\star}>c_1$ such that 
	$\liminf_{n\to\infty}\PR\big(\bar{\xi}_n(i_2(\lambda))-1 =j_{\star}\big)>0$.
	Let $n_{\star}$ be such that 
	\begin{align}\label{eqn:11}
		\inf_{n\geq n_{\star}}\PR\big(\bar{\xi}_n(i_2(\lambda)) =0\big)>0\ \text{ and }\ 
		\inf_{n\geq n_{\star}}\PR\big(\bar{\xi}_n(i_2(\lambda)) -1 =j_{\star}\big)>0\, .
	\end{align}
	Then for any $n\geq n_{\star}$ and $c_1\leq u\leq \delta n^{\alpha}$, 
	\begin{eq}
		&\E[|\zeta_1 - \zeta_2|^2\ind{|\zeta_1-\zeta_2| \leq u}] \geq \sum_{1\leq x\leq u} x\PR(\blue{x\leq |\zeta_1 - \zeta_2| \leq u}) \\
		&\hskip30pt 
		\geq \sum_{1\leq x\leq u/c_1} x\PR(\blue{x\leq \zeta_1\leq u})\PR(\zeta_2 =0) 
		\geq C u^{4-\tau},
	\end{eq}
	where \blue{the penultimate step uses the fact that if $ x\leq \zeta_1 \leq u$ and $\zeta_2 =0$, then $x\leq |\zeta_1 - \zeta_2| \leq u$}, and the final step follows using \eqref{eq:lower-bound-tail-xi} and the first inequality in \eqref{eqn:11}.
	Thus, \eqref{eq:esseen-inequality} yields, for $n\geq n_{\star}$ and any $l\geq 1$ satisfying $c_1\leq 2^{l+1}\leq\delta n^{\alpha}$,
	\begin{eq}
		\sup_{x\in I_l} \PR_x (\sigma_{nl} >t) 
		\leq 
		\sup_{x\in I_l} \PR_x (s_n(t) \in I_l) 
		\leq 
		Q(s_n(t),2^{l+1}) 
		\leq 
		\frac{C2^l}{(t2^{l(4-\tau)})^{\blue{1/2}}},
	\end{eq}
	which is at most $1/2$ by choosing $t = C 2^{l(\tau-2)}$ for some large constant $C>0$.

	Finally, for all $n\geq n_{\star}$ and $l\geq 1$ satisfying $2^{l+1}< c_1$,
	\[
	\sup_{x\in I_l} \PR_x (\sigma_{nl} >t)
	\leq
	\PR\big(\bar{\xi}_n(i_2(\lambda)) -1 \neq j_{\star}\big)^t
	\leq 
	\exp(-Ct)\, ,
	\]
	where the last step uses the second inequality in \eqref{eqn:11}.
	This in particular implies that $r_{nl}\leq C$ 
	for all $n\geq n_{\star}$ and $l\geq 1$ satisfying $2^{l+1}< c_1$.
	This completes the proof.
\end{proof}

We now decompose the possible values of the random walk~\eqref{eq:random-walk-tree} \blue{starting from $s_n(0) =1$} into different scales.
Recall that $I_l:=[2^{l-1} ,2^{l+1} )$.
At each time $t$, the scale of $s_n(t)$, denoted by $\scl(s_n(t))$, is an integer. 
Let $\scl(s_n(0)) =1$.
Suppose that $\scl(s_n(u)) = l$ for some $u>0$.
A change of scale occurs when $\bld{s}_n$ leaves $I_l$,
i.e., at time $T:= \inf\{t>u: s_n(t)\notin I_l\}$, and the new scale is given by $\scl(s_n(T)) = l'$, where $l'$ is such that $s_n(T)\in [2^{l'-1},2^{l'})$. 
Now, the next change of scale occurs at time $T':= \inf\{t>T: s_n(t)\notin I_{l'}\}$, and the scale remains the same until $T'$, i.e., $\scl(s_n(t)) = l'$ for all $T\leq t<T'$.
Define 
\begin{eq}
	H_{nl}(t):= \sum_{u \in [0,t), \ \scl(s_n(u))=l}\frac{1}{s_n(u)}, \quad \text{so that} \quad H_n(t) = \sum_{l\geq 1} H_{nl}(t).
\end{eq} 
\noindent Let $T_{nl}(t):= \#\{u\in [0,t): \scl(s_n(u))=l\}$, and note that 
\begin{eq}
	2^{l-1} H_{nl}(t)\leq T_{nl} (t)\leq 2^{l+1} H_{nl}(t).
\end{eq}
Therefore, for any $x>0$, 
\begin{equation}\label{eq:T-H-relation}
	\PR\Big(H_{nl}(\sigma)\geq \frac{xr_{nl}}{2^{l-1}}\Big)\leq \PR(T_{nl}(\sigma)\geq xr_{nl}).
\end{equation} 
The next lemma estimates $ \PR(T_{nl}(\sigma)\geq xr_{nl})$:

\begin{lemma}\label{lem:time-spent-l-ub} 
	For all $n\geq 1$, $l\geq 1$, and $x>0$,
	\begin{equation}
		\PR(T_{nl}(\sigma)\geq x r_{nl})\leq C2^{-l -C' x },
	\end{equation}for some absolute constants $C, C'>0$.
\end{lemma}
\begin{proof}
	Let us first show that for any $l\geq 2$,
	\begin{eq}\label{eq:T-nl-bound}
		\PR(T_{nl}(\sigma) \neq 0)\leq 2^{-(l-1)}.
	\end{eq}
	For any $t\geq 0$, let $\cF_t$ denote the sigma-field generated by $(\zeta_u)_{u =0}^t$, where we take $\zeta_0 =1$. 
	Note that if $T_{nl}(\sigma)\neq 0$, then  $s_n(u)$ hits $2^{l-1}$ before hitting zero. 
	For $H>1$, let $\gamma_H:= \min\{t: s_n(t) \geq H, \text{ or }s_n(t) =0\}$.
	Since $\E[\zeta_u-1]<0$ by \eqref{eq:BP-expt}, $(s_n(t))_{t\geq 0}$ is a supermartingale with respect to the filtration $(\cF_t)_{t\geq 0}$.
	Consequently, an application of the optional stopping theorem yields
	\begin{eq}
		H \PR(s_n(\gamma_H) \geq  H) \leq \E[s_n(\gamma_H)] \leq \E[s_n(0)] = 1,
	\end{eq}
	and therefore, 
	\begin{eq}\label{super-mg-hitting-time-bound}
		\PR(s_n(\gamma_H) \geq  H) \leq \frac{1}{H}.
	\end{eq}
	Thus, \eqref{eq:T-nl-bound} follows by taking $H = 2^{l-1}$ together with the fact that $T_{nl}(\sigma)\neq 0$ implies that $s_n(\gamma_H)\geq H$.

	Next, we define $U_n(t,[a,b))$--the number of upcrossings of an interval $[a,b)$ by $\bld{s}_n$ up to time $t$--to be the supremum of all integers $k$ such that there exist times $(u_j,t_j)_{j=1}^k$ satisfying $0\leq u_1<t_1<u_2<\dots<t_k\leq t$, and $s_n(u_j)<a<b\leq s_n(t_j)$ for all $j\in [k]$. 
	We will use the following simple fact (see \cite[Proposition 3.2]{A17}):
	for any positive integers $k, z, a, b$ with $ 0<z<a<b$,
	\begin{eq}\label{eq:upcrossing-inequality}
		\PR_z\big(U_n(\sigma,[a,b)) \geq k\big) \leq \Big(\frac{ a-1}{b}\Big)^k.
	\end{eq}
	Next define $\visit(l,t)$ to be the number of visits to scale $l$ upto time $t$, i.e., this is the supremum over $k\in \N$ such that one can find $(u_j,t_j)_{j=1}^k$ with $u_1<t_1< \dots <u_k<t_k \leq t$ satisfying $\scl(s_n(u_j))\neq l$ but $\scl(s_n(t_j)) = l$. 
	For the random walk $\bld{s}_n$ started at $s_n(0)=1$, we set $\visit(1,0) =1$ and $\visit(l,t) = 0$ \blue{if $\bld{s}_n$ does not enter scale $l$ before time $t$.} 
	Further, define $M_{nl} = \visit(l,\sigma)$ (the total number of visits to scale $l$) and $t_{jl} = \#\{t<\sigma:\scl(s_n(t))=l, \visit(l,t)=j \}$  (the time spent at scale $l$ during the $j$-th visit). \blue{Note that, if  $T_{nl} (\sigma) \neq 0$ occurs, then $M_{nl} \geq 1$, and}  $T_{nl} (\sigma) = \sum_{j=1}^{M_{nl}} t_{jl}$. 
	Thus, for any $m\geq 2$ and $x\in\bZ_{\geq 2}$,
	\begin{eq}\label{eq:Tnl-split-up}
		\PR\big(T_{nl} (\sigma)> 5x r_{nl} \big)=
		\PR\big(\sum_{j=1}^{M_{nl}} t_{jl}>\blue{5x} r_{nl}\big)
		\leq \PR(M_{nl}>m)+\PR\big(\sum_{j=1}^{m} t_{jl}>\blue{5x}r_{nl}\big).
	\end{eq}

	Now, $M_{nl}>m$ implies that $T_{nl}(\sigma)\neq 0$, and after the first visit to scale $l$, the walk comes back to scale $l$ at least $m$ times before hitting zero.
	In any of the subsequent visits, if $\bld{s}_n$ enters scale $l$ from below (this can only happen for $l\geq 3$), then that would imply an upcrossing of the interval $[2^{l-2},2^{l-1})$ has taken place.
	Otherwise, if $\bld{s}_n$ enters scale $l$ from above in any of the subsequent visits, then it must be the case that while leaving the scale $l$ during the previous visit, the walk went from scale $l$ to a higher scale.
	This yields an upcrossing of $[2^{l},2^{l+1})$.
	Therefore, 
	for any $l\geq 3$,
	$M_{nl}>m$ implies that $T_{nl}(\sigma)\neq 0$, and after the first visit to scale $l$ and before hitting zero, either at least $m/2$ many upcrossings of $[2^{l-2},2^{l-1})$ have taken place, or at least $m/2$ many upcrossings of $[2^{l},2^{l+1})$ have taken place.
	Thus, using \eqref{eq:T-nl-bound}, \eqref{eq:upcrossing-inequality}, and the strong Markov property, for any $l\geq 3$,
	\begin{eq}\label{Mnl-tail-bound}
		\PR(M_{nl}>m)
		\leq 
		\frac{C}{2^{l+m/2}}.
	\end{eq}

	Next, by the definition of $r_{nl}$ given right above Lemma~\ref{lem:r-nl-ub}, $\PR_z(t_{jl} > k r_{nl}) \leq 2^{-k}$ for any $z>0$, which implies that $\lfloor t_{jl}/r_{nl}\rfloor$ can be stochastically dominated by a  Geometric$(1/2)$ random variable.
	Using the strong Markov property, it follows that for any $z>0$, under $\PR_z$, $\sum_{j=1}^m \lfloor t_{jl}/r_{nl}\rfloor$ is stochastically dominated by $\sum_{i=1}^m g_i$, where $(g_i)_{i\geq 1}$ is an i.i.d.~collection of Geometric$(1/2)$ random variables.
	Thus, for any $z>0$,
	\begin{align*}
		\PR_z\bigg(\sum_{j=1}^mt_{jl}\geq (k+m)r_{nl}\bigg) 
		&\leq 
		\PR_z\bigg(\sum_{j=1}^m\Big\lfloor\frac{t_{jl}}{r_{nl}}\Big\rfloor > k\bigg)
		\leq 
		\PR\bigg(\sum_{i=1}^m g_i>k\bigg) \\
		&= 
		\PR(\mathrm{Bin}(k,1/2)<m)
		\leq 
		\e^{-(k-2m)^2/2k},
	\end{align*}
	for $2m\leq k$, where the last step follows using standard concentration inequalities such as \cite[Theorem~2.1]{JLR00}.
	Consequently, using \eqref{eq:T-nl-bound} and the strong Markov property,
	$
	\PR\big(\sum_{j=1}^mt_{jl}\geq (k+m)r_{nl}\big) 
	\leq 
	2^{-(l-1)}\cdot\e^{-(k-2m)^2/2k}
	$
	for $2m\leq k$.
	\blue{Combining this with \eqref{eq:Tnl-split-up} and \eqref{Mnl-tail-bound}, and taking $k = 4x$ and $m = x$ yields
		\begin{eq}
			\PR(T_{nl} (\sigma)> 5x r_{nl} ) \leq C 2^{-l-C'x}
		\end{eq}
		for any $l\geq 3$.
		The proofs for $l=1$ and $l=2$ follow similar steps.
		This completes the proof of Lemma~\ref{lem:time-spent-l-ub}.} 
\end{proof}

We are now ready to prove Proposition~\ref{prop:RW-hitting-estimate}.

\begin{proof}[Proof of Proposition~\ref{prop:RW-hitting-estimate}]
	Recall the definition of $\bld{s}_n$ from \eqref{eq:random-walk-tree}  starting from one, so that $s_n(0)=1$. 
	Fix $\delta>0$.
	We first estimate the probability of the event $\cB_n$ that $\bld{s}_n$  hits $\delta n^\alpha/2$ before hitting zero. 
	Let $\gamma:= \min\{t\colon s_n(t) \geq \delta n^{\alpha}/2, \text{ or }s_n(t) =0\}$. 
	By \eqref{super-mg-hitting-time-bound},
	\begin{eq}
		\PR(\cB_n) = \PR\Big(s_n(\gamma) \geq  \frac{\delta n^{\alpha}}{2}\Big) \leq \frac{2}{\delta n^{\alpha}}.
	\end{eq}
	\blue{Let $m := \max\{l\geq 1: 2^{l+1} \leq \delta n^{\alpha}\}$.
		On $\cB_n^c$, $H_{nl}(\sigma) = 0$ for $l>m$.
	}
	Thus, for any sequence of positive numbers $(b_l)_{l\geq 1}$, 
	\begin{eq}\label{eq:H-n-sigma-decompose}
		\PR\bigg(H_n(\sigma) \geq \sum_{l=1}^m \frac{b_l r_{nl}}{ 2^{l-1}}\bigg) &\leq \frac{2}{\delta n^{\alpha}} + \PR\bigg(H_n(\sigma) \geq \sum_{l=1}^m \frac{b_l r_{nl}}{ 2^{l-1}}, \text{ and } \cB_n^c \text{ occurs}\bigg)\\
		&\leq \frac{2}{\delta n^{\alpha}} + \PR\bigg(H_{nl}(\sigma) \geq \frac{b_l r_{nl}}{ 2^{l-1}} \text{ for some }1\leq l\leq m \bigg).
	\end{eq}
	Using \eqref{eq:T-H-relation} and Lemma~\ref{lem:time-spent-l-ub}, \eqref{eq:H-n-sigma-decompose} yields
	\begin{eq} \label{calc-simple-H-n-sigma}
		\PR\bigg(H_n(\sigma) \geq \sum_{l=1}^m \frac{b_l r_{nl}}{ 2^l}\bigg) \leq \frac{2}{\delta n^{\alpha}} + \sum_{l=1}^m \PR(T_{nl} (\sigma) \geq b_l r_{nl})\leq \frac{2}{\delta n^{\alpha}} + C\sum_{l=1}^m 2^{- l - C' b_l}. 
	\end{eq}
	Letting $b_l = \frac{1}{C'} (m-l+1 +2 \log_2 (m-l+1))$ for $1\leq l\leq m$, and using Lemma~\ref{lem:r-nl-ub},
	\begin{eq}\label{calc-sum-blrnl}
		\sum_{l=1}^m \frac{b_l r_{nl}}{2^{l-1}} &\leq C \sum_{l=1}^m \big(m-l+1 +2\log_2(m-l+1)\big) 2^{l(\tau-3)} \\
		&= C \sum_{j=1}^m \frac{(j+2\log_2 j) 2^{(m+1)(\tau-3)} }{2^{j(\tau-3)}}\leq C (\delta n^{\alpha}) ^{\tau - 3}  \sum_{j=1}^m \frac{(j+2\log_2 j) }{2^{j(\tau-3)}}\\
		& \leq C n^{\eta} \delta^{\tau -3},
	\end{eq}where we have used $\sum_{j=1}^\infty  \frac{(j+2\log_2 j) }{2^{j(\tau-3)}} <\infty$ in the last step; 
	the bound in \eqref{calc-sum-blrnl} holds for all $n\geq n_{\star}$, where $n_{\star}$ is as in Lemma~\ref{lem:r-nl-ub}.
	Also, 
	\begin{eq}\label{calc-prob-ub}
		\sum_{l=1}^m 2^{-l- C'b_l} = \sum_{l=1}^m 2^{-(m+1)} (m-l+1)^{-2} \leq \frac{4}{\delta n^{\alpha}} \sum_{l=1}^\infty \frac{1}{l^2}.
	\end{eq}
	Thus, the claim in Proposition~\ref{prop:RW-hitting-estimate} follows for $n\geq n_{\star}$ by combining \eqref{calc-sum-blrnl} and \eqref{calc-prob-ub} with \eqref{calc-simple-H-n-sigma}.
	We conclude that the claimed bound holds for $n\geq N_{\lambda}$ by choosing a larger constant $C$ on the right side of \eqref{eqn:2}.
\end{proof}

\subsection{Proof of Proposition~\ref{prop:diamter-small-comp}} 
Let us now complete the proof of Proposition~\ref{prop:diamter-small-comp} using Proposition~\ref{prop:coupling-uppperbound} and Theorem~\ref{lem:boundary-small-prob}. 
\blue{We take $K_n$ as in Lemma~\ref{lem:technical} so that the results in Section~\ref{sec:height-vs-rw} hold. }
Note that these bounds work for $i_2(\lambda)\leq i\leq K_n$, and we will use path counting arguments from \cite{J09b,BDHS17} to bound the diameter for $i>K_n$.
Define $\sC_{ \mathrm{res}} (i) $ to be the connected component containing vertex $i$ in the graph $\cG_n^{\sss > i-1}=\CM \setminus [i-1]$.
Note that if $\Delta^{\sss >K}>\varepsilon n^{\eta}$, then there exists
a path of length $\varepsilon n^\eta$ in $\CM$ avoiding all the vertices in $[K]$. 
Suppose that the minimum index among vertices on that path is $i_0$.
Then $\diam (\sC_{ \mathrm{res}} (i_0) )>\varepsilon n^{\eta}$.
Therefore, $\Delta^{\sss >K}>\varepsilon n^{\eta}$ implies that either there exists $i\in (K,K_n)$ satisfying $\diam (\sC_{ \mathrm{res}} (i) )>\varepsilon n^{\eta}$, or $\mathrm{diam}(\CM\setminus [K_n]) > \varepsilon n^{\eta}$.
We will use the following lemma first to complete the proof of Proposition~\ref{prop:diamter-small-comp} and prove the lemma subsequently:
\begin{lemma}\label{lem:diam-bound-subcritical}
	Under \rm Assumptions~\ref{assumption1} and \ref{assumption-extra}, for any $\varepsilon>0$,  $\lim_{n\to\infty}\PR(\mathrm{diam}(\CM\setminus [K_n]) > \varepsilon n^{\eta}) = 0$, \blue{where  $K_n$ as in Lemma~\ref{lem:technical}.}
\end{lemma}

\begin{proof}[Proof of Proposition~\ref{prop:diamter-small-comp}]
	As defined earlier around \eqref{eqn:668}, $\partial_i(r)$ denotes the number of vertices at distance $r$ from the vertex $i$ in the graph $\cG_n^{\sss > i-1}$. 
	Recall the definition of $\bar{\partial}$ in Proposition~\ref{prop:coupling-uppperbound}.
	Thus, Proposition~\ref{prop:coupling-uppperbound} and Theorem~\ref{lem:boundary-small-prob} together with Lemma~\ref{lem:diam-bound-subcritical}, yield that
	\begin{eq}\label{eq:union-bound-diameter}
		\prob{ \Delta^{\sss >K}>\varepsilon n^{\eta} } &\leq \sum_{i\in (K,K_n)}\PR(\bar{\partial}_i(\varepsilon n^{\eta}\red{/2})\neq \varnothing) + \PR(\mathrm{diam}(\CM\setminus [K_n]) > \varepsilon n^{\eta})\\
		&\leq C\sum_{i\in (K,K_n)} \Big(\frac{d_i}{n^{\alpha}}\Big)\e^{-\varepsilon\beta_i^n/4} +o(1),
	\end{eq}
	where the last line tends to zero if we first take $n\to\infty$ and then take $K\to\infty$ using  Assumption~\ref{assumption-extra} \blue{and Lemma~\ref{lem:technical} below}. 
	Thus the proof of Proposition~\ref{prop:diamter-small-comp} follows.
\end{proof}

\begin{proof}[Proof of Lemma~\ref{lem:diam-bound-subcritical}]
	Let $\bld{d}':=(d_i'\, ;\, i\in [n]\setminus [K_n])$, where $d_i'$ denotes the degree of $i$ in $\CM\setminus [K_n]$.
	Note that $\CM\setminus [K_n]$ is again distributed as a configuration model conditionally on its degree sequence $\bld{d}'$, with the criticality parameter 
	\begin{eq}
		\nu'_n = \frac{\sum_{i>K_n}d_i'(d_i'-1)}{\sum_{i>K_n} d_i'} \leq \frac{\sum_{i>K_n}d_i(d_i-1)}{\ell_n - 2 \sum_{i=1}^{K_n} d_i} \leq 1 - R_n n^{-\eta}, \quad R_n = \omega( \log n),
	\end{eq}
	where the penultimate step follows using $d_i'\leq d_i$ and $\ell_n' := \sum_{i>K_n} d_i' = \ell_n - 2 \sum_{i=1}^{K_n} d_i$, and the last bound follows from 
	\blue{the definition of $K_n$ given in Lemma~\ref{lem:technical}~(ii)} and an argument identical to that in \eqref{eq:computation-mean-BP}.
	Let $\PR'(\cdot)$ denote the probability measure conditionally on $\bld{d}'$.
	We will use path-counting techniques for subcritical configuration models. 
	An argument similar to the one given in \cite[Lemma 6.1]{J09b} shows that for any $l\geq 1$, conditional on $\bld{d}'$, the expected number of paths of length $l$ starting from vertex $i$ is at most 
	\[
	\frac{\ell_n'd_i' (\nu_n')^{l-1}}{\ell_n'-2l+3} \leq \ell_n'^2 (\nu_n')^{l-1}.
	\]
	Thus, for any $i>K_n$,
	\begin{eq} \label{eq:path-counting-tail}
		\PR'(\exists \text{ path of length at least }\varepsilon n^\eta \text{ from } i \text{ in }\CM\setminus [K_n]) \leq C \ell_n'^2\sum_{l>\varepsilon n^{\eta}} (\nu_n')^l,
	\end{eq}
	Thus, for $i>K_n$,  the probability in \eqref{eq:path-counting-tail} is at most 
	\begin{eq} 
		Cn^2(1-R_n n^{-\eta})^{\varepsilon n^{\eta}}/(R_n n^{-\eta})
		\leq 
		Cn^{2+\eta} \e^{-\varepsilon R_n} 
		= 
		o(1/n),
	\end{eq}
	\blue{since $R_n\gg \log{n}$.} 
	Therefore, 
	\begin{eq}
		\PR'(\exists i>K_n: \exists \text{ path of length at least }\varepsilon n^\eta \text{ from } i \text{ in }\CM\setminus [K_n]) = o(1),
	\end{eq}and the proof of Lemma~\ref{lem:diam-bound-subcritical} follows.
\end{proof}

\section{Verification of the assumptions for  percolated degrees: Proof of Theorem~\ref{cor:GLM-percoltion}} \label{sec:perc-degrees}
Let $\cG_n$ denote the graph obtained by performing percolation with edge retention  probability $p_c(\lambda)$ (defined in \eqref{eq:critical-window-defn}) on $\CM$.
Let $\bld{d}^p=(d_i^p)_{i\in [n]}$ denote the degree sequence of $\cG_n$. 
By \cite[Lemma 3.2]{F07}, the law of $\cG_n$, conditionally on $\bld{d}^p$, is the same as the law of $\rCM_n(\bld{d}^p)$. 
Thus, it is enough to show that if the original degree sequence $(\bld{d}_n \, ,\, n\geq 1)$ satisfies Assumptions~\ref{assumption1}(i),~\ref{assumption1}(ii)~and~\ref{assumption-extra},
then we can construct $(\bld{d}^p\, ,\, n\geq 1)$ on the same probability space so that  Assumption~\ref{assumption1}, \eqref{defn:criticality}, and Assumption~\ref{assumption-extra} are satisfied almost surely (with possibly different parameters), since then the claim in Theorem~\ref{cor:GLM-percoltion} will follow from Theorem~\ref{thm:gml-bound}.

\blue{First, note that $\E[d_i^p] = d_ip_c(\lambda) (1+o(1))$. Also, given $\CM$, changing the status of an edge (deleted or retained) can change $d_i^p$ by at most $2$ when the edge is incident to $i$. There are at most $d_i$ choices for such an edge. Thus, the bounded difference inequality \cite[Corollary 2.27]{JLR00} implies}, for each fixed $i\geq 1$, and for any $\varepsilon>0$,
\begin{gather}
	\PR\big( |d_i^p - d_i p_c(\lambda)| > \varepsilon d_i p_c(\lambda)\big) \leq 2\e^{-\frac{\varepsilon^2}{4} d_i p_c^2(\lambda)}\, . \label{hub-perc}
\end{gather}
In particular, for each $i\geq 1$, $n^{-\alpha} d_i^p \weakc \theta_i/\nu$ as $n\to\infty$, which verifies Assumption~\ref{assumption1}~(i).

Next, let $M_r^p = \sum_{i\in [n]} (d_i^p)_{r}$ and $M_r = \sum_{i\in [n]} (d_i)_{r}$, where $(x)_r:= x(x-1)\cdots (x-r+1)$. 
To verify the moment conditions in  Assumption~\ref{assumption1}~(ii), 
note that \eqref{eqn:670} holds for $\bld{d}^p$ since \blue{$\sum_{i>K} (d_i^p)_3 \leq \sum_{i>K} (d_i)_3 $.} 
We will show that 
\begin{eq}\label{moment-convergence}
	M_1^p =(1+\OP(n^{-1/2}))p_c(\lambda)M_1\ \ \text{ and }\ \
	M_2^p =(1+\OP(n^{\frac{3\alpha}{2}-1}))p_c(\lambda)^2M_2.
\end{eq}
Using \eqref{moment-convergence}, the first and second moment assumptions in Assumption~\ref{assumption1}~(ii) holds for the percolated degree sequence. 
The estimate \eqref{moment-convergence} also shows that \eqref{defn:criticality} holds. Indeed, 
\begin{eq}
	\frac{M_2^p}{M_1^p} = \frac{p_c(\lambda)M_2}{M_1} (1+ \OP(n^{\frac{3\alpha}{2}-1})) = 1+\nu\lambda n^{-\eta} +o(n^{-\eta}),
\end{eq}
where the last step follows using \eqref{eq:defn-super-crit}, \eqref{eq:critical-window-defn}, and the fact that $-1+3\alpha/2< -1+2\alpha = -\eta$.

It remains to prove \eqref{moment-convergence}. 
Since $\frac{1}{2}\sum_{i\in [n]} d_i^p $ has a binomial distribution with parameter $\ell_n/2$ and $p_c(\lambda)$, the first asymptotics follows from Chebyshev's inequality. 
For the asymptotics of $M_2^p $, we use the following construction of $\cG_n$ from~\cite{F07}. 
\begin{algo}\label{algo:perc-degrees}\normalfont
	$\bld{d}^p = (d_i^p)_{i\in [n]}$ can be generated as follows:
	\begin{enumerate}
		\item[(S0)] Sample $R_n \sim \mathrm{Bin}(\ell_n/2, p_c(\lambda))$.
		\item[(S1)] Conditionally on $R_n$, sample a uniform subset of $2R_n$ half-edges from the set of $\ell_n$ half-edges. Let $I_j^{(i)}$ denote the indicator that $j$-th half-edge of $i$ is selected. Then $d_i^p = \sum_{j=1}^{d_i} I_j^{(i)}$ for all $i\in [n]$.
	\end{enumerate}
\end{algo}
\noindent Using the above construction, note that 
\begin{eq}
	M_2^p = \sum_{i\in [n]} \sum_{\substack{1\leq j_1\neq j_2\leq d_i}}I_{j_1}^{(i)}I_{j_2}^{(i)}.
\end{eq}
Let $\PR_1(\cdot) = \PR(\cdot \vert R_n)$ and similarly define $\E_1[\cdot]$, $\mathrm{Var}_1(\cdot)$, and $\mathrm{Cov}_1(\cdot,\cdot)$.
Then, 
\begin{eq}\label{M-2-p-cond-exp}
	\E_1[M_2^p] &= \sum_{i\in [n]} \sum_{\substack{1\leq j_1\neq j_2\leq d_i}} \PR_1 (I_{j_1}^{(i)}=1, I_{j_2}^{(i)} = 1)  = \sum_{i\in [n]} \sum_{\substack{1\leq j_1\neq j_2\leq d_i}} \frac{{\ell_n-2 \choose 2R_n-2}}{{\ell_n \choose 2R_n}}\\
	&=\sum_{i\in [n]} \sum_{\substack{1\leq j_1\neq j_2\leq d_i}} \frac{2R_n(2R_n-1)}{\ell_n(\ell_n-1)} = (1+\OP(n^{-1/2}))p_c(\lambda)^2 M_2,
\end{eq}
where the last step follows using $R_n = (1+\OP(n^{-1/2})) p_c(\lambda)\ell_n/2$.

\blue{Next, recall that a collection of random variables $(X_1, \dots, X_t)$ is called negatively associated if for every index set $I\subset [k]$, 
	\begin{eq}\label{defn:neg-association}
		\mathrm{Cov}\big(f(X_i,i\in I), g(X_i,i\in I^c)\big) \leq 0,
	\end{eq}
	for all functions $f: \R^{|I|} \mapsto \R$ and $g: \R^{t-|I|} \mapsto \R$ that are component-wise non-decreasing (\cite[Definition 3]{Dubhashi1996}).} 
Then, conditionally on $R_n$, $I_j^{(i)}$, $j = 1,\dots, d_i$, $i\in [n]$ are negatively associated (cf.~\cite[Theorem 10]{Dubhashi1996}), which yields \blue{the almost sure bound}
\begin{eq}
	\mathrm{Var}_1(M_2^p) &\leq \sum_{i\in [n]} \sum_{\substack{1\leq j_1\neq j_2\leq d_i}} \mathrm{Var}_1(I_{j_1}^{(i)}I_{j_2}^{(i)})+\sum_{i\in [n]} \sum_{\substack{1\leq j_1\neq j_2\leq d_i\\
			1\leq j_3\neq j_4\leq d_i\\
			|\{j_1,j_2\}\cap \{j_3,j_4\}| = 1}} \mathrm{Cov}_1(I_{j_1}^{(i)}I_{j_2}^{(i)},I_{j_3}^{(i)}I_{j_4}^{(i)}),
\end{eq}
\blue{since the contribution of $|\{j_1,j_2\}\cap \{j_3,j_4\}| = 0$ can be ignored due to negative association. 
	Also, $\mathrm{Var}_1(I_{j_1}^{(i)}I_{j_2}^{(i)})\leq 1$ and $|\mathrm{Cov}_1(I_{j_1}^{(i)}I_{j_2}^{(i)},I_{j_3}^{(i)}I_{j_4}^{(i)})| \leq (\mathrm{Var}_1(I_{j_1}^{(i)}I_{j_2}^{(i)})\mathrm{Var}_1(I_{j_3}^{(i)}I_{j_4}^{(i)}))^{1/2} \leq 1$. 
	Therefore, 
	\begin{eq}
		\mathrm{Var}_1(M_2^p)\leq \sum_{i\in [n]}d_i^2+4 \sum_{i\in [n]}d_i^3 = O(n^{3\alpha}).
	\end{eq}
}
Thus, for any $A>0$,
\begin{eq}\label{M-2-p-chebyshev}
	&\PR \big( |M_2^p - \E_1[M_2^p]| > A n^{3\alpha/2} \big) = \E\big[\PR_1 \big( |M_2^p - \E_1[M_2^p]| > A n^{3\alpha/2} \big) \big]\leq \frac{\E\big[\mathrm{Var}_1(M_2^p)\big]}{A^2n^{3\alpha}} ,
\end{eq}
which can be made arbitrarily small by choosing $A>0$ large. 
Thus, we conclude the asymptotics of $M_2^p$ in \eqref{moment-convergence} by using \eqref{M-2-p-cond-exp} and \eqref{M-2-p-chebyshev}.

Finally, we need to show convergence \blue{in distribution} of empirical measure of $\bld{d}^p$ to finish verifying Assumptions~\ref{assumption1}~(ii)~\blue{ and~(iii)}. Let $n_k^p= \#\{i:d_i^p = k\}$, and $n_{\geq k} ^p = \sum_{r\geq k} n_{r}^p$. 
It suffices to show that 
\begin{eq}\label{empirical-perc}
	\frac{n_{\geq k}^p}{n} \weakc \PR(D^p \geq  k) \ \ \text{ for all } k\geq 1, 
\end{eq}
where $D^p$ satisfies $\big(D^p\mid D = l\big) \sim \mathrm{Bin} (l,1/\nu)$, for all $l\geq 1$. 
Let $V_n$ be a uniformly chosen vertex and $D_n^p= d_{V_n}^p$ and $D_n = d_{V_n}$. 
By the construction in Algorithm~\ref{algo:perc-degrees},
\begin{eq}
	\PR(D_n^p = k \mid D_n = l, R_n) &= \frac{{l\choose k} {\ell_n - l \choose 2R_n - k}}{{\ell_n \choose 2R_n}} = (1+o(1)) {l\choose k} \bigg(\frac{2R_n}{\ell_n}\bigg)^k\bigg(1-\frac{2R_n}{\ell_n}\bigg)^{l-k}\\
	& = (1+\oP(1)) {l\choose k} \bigg(\frac{1}{\nu}\bigg)^k\bigg(1-\frac{1}{\nu}\bigg)^{l-k},
\end{eq}
where in the final step we have used that $R_n =  p_c(\lambda)\ell_n/2  (1+\oP(1))$ and $p_c(\lambda) = \nu^{-1} (1+o(1))$.
Thus, 
\begin{eq}\label{exp-perc-empricial}
	\E\bigg[\frac{n_k^p}{n}\ \Big\vert\ R_n\bigg] = \PR(D_n^p = k \mid R_n) \weakc \PR(D^p = k). 
\end{eq}
\blue{Moreover, note that $d_i^p = \sum_{j=1}^{d_i} I_j^{(i)}$, and the definition of negative association in \eqref{defn:neg-association} allows us to conclude negative correlation between increasing functions of $I_j^{(i)}$ that depend on disjoint sets of indices. Therefore,} 
\begin{eq}
	\mathrm{Var}\big(n_{\geq k}^p \mid R_n\big) 
	= 
	\mathrm{Var}\big(\sum_{i\colon d_i\geq k} \1_{\{d_i^p\geq k\} }\ \Big\vert\  R_n\big) 
	\leq 
	\sum_{i\colon d_i\geq k} \mathrm{Var} \big(\1_{\{d_i^p\geq k\} } \mid R_n\big)  
	\leq 
	n,
\end{eq}
and thus $\E[\mathrm{Var}(n_{\geq k}^p/n \mid R_n)] = O(1/n)$. 
This together with \eqref{exp-perc-empricial} yields \eqref{empirical-perc}.

We finally verify that Assumption~\ref{assumption-extra} holds with high probability. 
Let the constants $c_1$ and $c_0$ be as in Assumption~\ref{assumption-extra}.
Let $c_1':=9 c_1\nu$.
It is enough to show that there exists a deterministic constant \blue{$c_0'>0$} such that
\begin{eq}\label{tail-percolated}
	\pr\Big(
	\frac{1}{n}\sum_{i\in [n]} d_i^p\blue{\1\{l<d_i^p\leq c_1' l \}} \geq \frac{c_0'}{l^{\tau-2}} 
	\  \ \text{ for all }\ \ 1\leq l \leq  d_1/c_1'
	\Big)
	\to 
	1
	\ \ \text{ as }\ \ n\to\infty.
\end{eq}
\blue{
	Write $\sum_{*}$ for $\sum_{i: 8 l < d_i p_c(\lambda) \leq 8c_1 l }$.	
	Then, for $1\leq l \leq d_1/c_1'$ and all large $n$, 
	\begin{eq}\label{eqn:23}
		\frac{1}{n}\sum_{i\in [n]} \E\big[d_i^p\1\{l<d_i^p\leq c_1' l \}\big] 
		&
		\geq \frac{1}{n}\sum\displaystyle_*\ \E\big[d_i^p\1\{l<d_i^p\leq c_1' l \}\big] \\
		&=
		\frac{1}{n}\sum\displaystyle_*\  \E\big[d_i^p\1\{l<d_i^p\} \big],
	\end{eq}
	where the last step uses the fact that when $d_i p_c(\lambda)\leq 8c_1 l $, we have 
	$d_i^p \leq d_i \leq 8c_1 l  /p_c(\lambda) \leq c_1' l$ for all large $n$. 
	Let $X_i \sim \mathrm{Bin} (\lfloor d_i/2\rfloor , p_c(\lambda))$. 
	Then $d_i^p$ is stochastically larger than $X_i$. 
	Using \eqref{eqn:23}, we see that for $1\leq l \leq d_1/c_1'$ and all large $n$, 
	\begin{eq}\label{expt-perc-tail}
		&
		\frac{1}{n}\sum_{i\in [n]} \E\big[d_i^p\1\{l<d_i^p\leq c_1' l \}\big] 
		\geq
		\frac{1}{n}\sum\displaystyle_*\  \E\big[X_i\cdot\1\{l<X_i\} \big]
		\\
		&\hskip25pt
		\geq
		\frac{1}{n}\sum\displaystyle_*\  \E[X_i]\PR(X_i>l) 
		\geq 
		\frac{1}{n}\sum\displaystyle_*\  \E[X_i] \PR\big( X_i \geq \floor{\lfloor d_i/2\rfloor p_c(\lambda)} \big) 
		\\
		&\hskip50pt
		\geq
		\frac{1}{n}\sum\displaystyle_*\  \E[X_i] \cdot\frac{1}{2}
		\geq 
		\frac{C}{n}\sum\displaystyle_*\  d_i 
		\geq 
		\frac{C'}{l^{\tau-2}} \, ,
	\end{eq}
	where the third step uses $l< d_i p_c(\lambda) /8 \leq \floor{\lfloor d_i/2\rfloor p_c(\lambda)}$, and the final step follows using Assumption~\ref{assumption-extra}}.

Now let $F_1 := \sum_{i\in [n]} d_i^p\1\{l<d_i^p\leq c_1' l \}$ and $F_2 := \E[F_1\vert \CM ]$.
We will apply the bounded difference  inequality from  \cite[Corollary 2.27]{JLR00}.
Given the graph $\CM$, if we keep one extra edge in the percolated graph, then $F_1$ can change by at most $2c_1' l$. 
Thus, for any $\varepsilon>0$,
\begin{eq}\label{prob-tail-conc-1}
	\PR\Big(|F_1 - F_2|> \frac{n\varepsilon }{l^{\tau-2}}\ \Big\vert\ \CM\Big)\leq 2\exp\bigg(-
	\frac{n^2\varepsilon^2}{l^{2(\tau-2) }(2c_1' l)^2\cdot \frac{\ell_n}{2} 
	}\bigg) \leq 2\e^{-C\varepsilon^2 n l^{-2(\tau-1)}} .
\end{eq}
Also, we can apply concentration inequalities such as \cite[Lemma 2.5]{BCDS18} to conclude that 
\begin{eq}\label{prob-tail-conc-2}
	\PR\Big(|F_2 - \E[F_2]|> \frac{n\varepsilon }{l^{\tau-2}}\ \Big)\leq 2\e^{-C\varepsilon^2 n l^{-2(\tau-1)} 
	}. 
\end{eq}
Combining \eqref{prob-tail-conc-1} and  \eqref{prob-tail-conc-2} together with  \eqref{expt-perc-tail} shows that there exists an $\varepsilon_0>0 $ such that \eqref{tail-percolated} holds if we replace 
``for all $1\leq l\leq d_1/c_1'$'' by
``for all $1\leq l\leq n^{\varepsilon_0}$.'' 
For $l\geq n^{\varepsilon_0}$, we use \eqref{hub-perc} together with a union bound to complete the proof of \eqref{tail-percolated}.

\appendix

\section{A technical lemma}
\begin{lemma}\label{lem:technical}
	Let $\beta_i^n = n^{-2\alpha} \sum_{j=1}^{i-1} d_j^2$. Then {\rm Assumption~\ref{assumption1}(i), \eqref{eqn:670}, and Assumption~\ref{assumption-extra}} imply the following: 
	\begin{enumeratei}
		\item For all $\varepsilon>0$, 
		\begin{eq}\label{eq:extra-assumption-2}
			\lim_{K\to\infty} \limsup_{n\to\infty} \sum_{i > K} \Big(\frac{d_i}{n^\alpha}\Big) \times \e^{-\varepsilon \beta_i^n} = 0.
		\end{eq}
		\item There exists a sequence $(K_n)_{n\geq 1}$ with $K_n\to\infty$, and $K_n = o(n^\alpha)$ such that $\beta_{K_n}^{n} >  \log^3 n $ for all large $n$.
	\end{enumeratei}
\end{lemma}

\begin{proof}
	We will use $C_0,C_1,\dots$ etc. as generic notation for positive constants that do not depend on $n$.
	Recall Assumption~\ref{assumption-extra}. 
	Let $\theta_{i,n} := n^{-\alpha}d_i$, $i\in [n]$. We first claim that Assumption~\ref{assumption-extra} implies 
	\begin{eq}\label{eq:deg-power-lb}
		\min_{2\leq i\leq n} \theta_{i,n}^{\tau -2} \sum_{j = 1}^{i-1}\theta_{j,n} \geq C_0.
	\end{eq}
	To see this, let $1=i_1<i_2<i_3<\ldots$ be the indices such that 
	$d_{i_{k-1}}=d_{i_{k-1} +1}=\ldots=d_{i_k -1}>d_{i_k}$ for $k\geq 2$. 
	Then for $k\geq 2$,
	\[
	\frac{1}{\ell_n}\sum_{j=1}^{i_{k}-1}d_j
	=
	\PR(D_n^* > d_{i_k}) 
	\geq 
	\PR\big(d_{i_k}<D_n^*\leq c_1 d_{i_k} \big) 
	\geq 
	c_0 (d_{i_k})^{-(\tau -2)} ,
	\]
	and consequently,
	\begin{eq}\label{i-k-tail}
		\min_{k} \theta_{i_k,n}^{\tau -2} \sum_{j = 1}^{i_k-1}\theta_{j,n} \geq C_0.
	\end{eq}
	If $i_k>i\geq i_{k-1}$, then 
	$\theta_{i,n}^{\tau -2} \sum_{j = 1}^{i-1}\theta_{j,n} 
	=
	\theta_{i_{k-1},n}^{\tau -2} \sum_{j = 1}^{i-1}\theta_{j,n} 
	\geq 
	\theta_{i_{k-1},n}^{\tau -2} \sum_{j = 1}^{i_{k-1}-1}\theta_{j,n}$. 
	Thus we conclude \eqref{eq:deg-power-lb} from \eqref{i-k-tail}. 

	Next, define 
	\begin{eq}
		f_n (x):= 
		\begin{cases}
			\frac{1}{\theta_{i+1,n}}\, , & \quad \text{if } \sum_{j=1}^{i-1} \theta_{j,n} \leq x < \sum_{j=1}^{i} \theta_{j,n} \text{ for some }i\in[n-1], \\
			0\, , & \quad \text{if } x\geq \sum_{j=1}^{n-1} \theta_{j,n},
		\end{cases}
	\end{eq}
	and 
	\begin{eq}
		g_n (x):= 
		\begin{cases}
			\sum_{j=1}^i \theta_{j,n}^2\, , & \quad \text{if } \sum_{j=1}^{i-1} \theta_{j,n} \leq x < \sum_{j=1}^{i} \theta_{j,n} \text{ for some }i\in[n], \\
			0\, , & \quad \text{if } x\geq \sum_{j=1}^{n} \theta_{j,n}\, . 
		\end{cases}
	\end{eq}
	Since $\sum_{j=1}^i\theta_{j,n}\leq 2\sum_{j=1}^{i-1}\theta_{j,n}$ for $2\leq i\leq n$,
	we have, using \eqref{eq:deg-power-lb}, $\theta_{i+1,n}^{\tau -2} \sum_{j = 1}^{i-1}\theta_{j,n} \geq C_0/2$.
	Therefore, 
	$f_n(x)^{-(\tau -2) } \times 2x \geq C_0$ for any $\theta_{1,n}\leq x < \sum_{j=1}^{n-1} \theta_{j,n}$,
	and consequently, 
	\begin{eq}
		f_n(x) \leq C_1 x^{\frac{1}{\tau-2}} \quad \text{for} \quad \theta_{1,n}\leq x < \sum_{j=1}^{n-1} \theta_{j,n}.
	\end{eq}
	Next, for $i\in[n-1]$, 
	\begin{eq}
		\sum_{j=1}^i \theta_{j,n}^2 \geq \sum_{j=1}^i \theta_{j,n} \theta_{j+1,n} &= \theta_{1,n} \theta_{2,n} + \int_{\theta_{1,n}}^{\sum_{j=1}^i \theta_{j,n}} \frac{\dif x}{f_{n}(x)}\\
		&\geq C_2\int_{\theta_{1,n}}^{\sum_{j=1}^i \theta_{j,n}} \frac{\dif x}{x^{1/(\tau-2)}}\geq C_3 \bigg(\sum_{j=1}^i \theta_{j,n}\bigg)^{\frac{\tau-3}{\tau-2}} - C_4. 
	\end{eq}
	Therefore, 
	\begin{eq}\label{lb-g-n-x}
		g_n(x) \geq C_3 x^{\frac{\tau-3}{\tau-2}} - C_4 \quad \text{for} \quad 0\leq x < \sum_{j=1}^{n-1} \theta_{j,n}.
	\end{eq}
	Now, 
	\begin{eq}
		\sum_{i=K}^{n-1} \theta_{i,n} \e^{-\varepsilon \sum_{j=1}^i  \theta_{j,n}^2} = \int_{\sum_{j=1}^{K-1} \theta_{j,n}}^{ \sum_{j=1}^{n-1} \theta_{j,n} } \e^{-\varepsilon g_n(x)} \dif x \leq C_5\int_{\sum_{j=1}^{K-1} \theta_{j,n}}^{ \infty} \e^{-\varepsilon C_3 x^{\frac{\tau-3}{\tau-2}}} \dif x,
	\end{eq}
	and the above integral is finite for each fixed $K\geq 1$. 
	By Assumption~\ref{assumption1}(i), $\sum_{j=1}^{K-1} \theta_{j,n} \to \sum_{j=1}^{K-1} \theta_{j}$ as $n\to\infty$, which diverges if we take $K\to\infty$. Thus, the proof of \eqref{eq:extra-assumption-2} follows.

	We next prove Lemma~\ref{lem:technical}(ii). 
	Let $K_n:=\lceil n^{\alpha/2}\rceil$.
	Suppose that $\beta^n_{K_n} \leq \log ^3 n$. 
	Using \eqref{lb-g-n-x}, it follows that 
	\begin{eq}
		\log^3 n \geq \beta^n_{K_n} \geq C_3 \bigg(\sum_{j=1}^{K_n} \theta_{j,n}\bigg)^{\frac{\tau-3}{\tau-2}} - C_4,
	\end{eq}and an application of \eqref{eq:deg-power-lb} yields
	\begin{eq}
		C_4 + \log ^3 n \geq C (\theta_{K_n+1,n})^{-(\tau-3)} \implies \theta_{K_n,n} \geq \frac{C'}{(\log n)^{\frac{3}{\tau-3}}}\, . 
	\end{eq}
	Therefore, $\sum_{i=1}^{K_n} \theta_{i,n}^3 \geq C'^3  K_n (\log n)^{-9/(\tau-3)}$. 
	Thus, if $\beta^n_{K_n} \leq \log ^3 n$ for infinitely many $n$, then 
	\begin{eq}
		\liminf_{n\to\infty}n^{-3\alpha} \sum_{i\in [n]} d_i^3 
		\geq 
		\liminf_{n\to\infty} \sum_{i=1}^{K_n} \theta_{i,n}^3 
		= 
		\infty\, ,
	\end{eq}
	which leads to a contradiction as Assumption~\ref{assumption1}(i) and \eqref{eqn:670} imply that $\sup_{n}n^{-3\alpha}\sum_{i\in [n]} d_i^3<\infty$. 
	Thus the claim in Lemma~\ref{lem:technical}~(ii) also follows. 
\end{proof}

\section{Degree sequence satisfying compactness criterion}\label{sec:appendix-comapctness}
In this section, we prove Proposition~\ref{prop:deg-compact}.

\vskip5pt

\noindent{\bf Proof of Proposition~\ref{prop:deg-compact}.}
Define $\bld{d}^{\sss (1,n)}:=(d_i^{\sss (1,n)})_{i\in [n]}$ with $d_i^{\sss (1,n)}: = \lceil n^{\alpha} \theta_i \rceil$ for $i\in [n]$. 
Let $\bld{d}^{\sss (2,n)} = (d_i^{\sss (2,n)})_{i\in [n]}$ be such that, for some $0<K_1<K_2<\infty$,
\begin{eq}\label{eq:d-2n}
	K_1\Big( \frac{n}{i} \Big)^{\alpha} \leq d_i^{\sss (2,n)} \leq K_2\Big( \frac{n}{i} \Big)^{\alpha}, \quad \text{ for }i \in [n], 
\end{eq}
and Assumption~\ref{assumption1}(ii) and \eqref{eq:defn-super-crit} are satisfied. 
The idea is to change the high-degree vertices of $\bld{d}^{\sss (2,n)}$ by those of $\bld{d}^{\sss (1,n)}$. To this end,
let
\[
i_{\sss (1,n)}:= \max\big\{i\geq 1\colon d_i^{\sss (1,n)} \geq (\frac{n}{\log n})^{\alpha}\big\}
\ \ \text{ and }\ \ 
i_{\sss (2,n)}:= \max\big\{i\geq 1\colon d_i^{\sss (2,n)} \geq (\frac{n}{\log n})^{\alpha}\big\}\, .
\]
For two finite sequences $(x_i)$ and $(y_j)$, we write $\texttt{Sort-Merge}((x_i),(y_j))$ as the sequence obtained by concatenating $(x_i)$ and $(y_j)$ and then sorting the sequence in a nonincreasing order.
We define
\begin{eq}\label{defn:merged-degree}
	\bld{d}^{\sss (n)}
	=
	(d_i^{\sss (n)})
	:=
	\texttt{Sort-Merge} \Big((d_i^{\sss (1,n)})_{i=1}^{i_{\sss (1,n)}}, (d_i^{\sss (2,n)})_{i=i_{\sss (2,n)}+1}^{n}\Big).
\end{eq}
Note that $i_{\sss (1,n)} \to \infty$. 
Also,
\begin{eq}
	\infty
	>
	\sum_{i=1}^\infty \theta_i^3 
	\geq 
	\sum_{i=1}^{i_{\sss (1,n)}} \theta_i^3 
	\geq 
	i_{\sss (1,n)} \theta_{i_{\sss (1,n)}}^3 
	\geq 
	i_{\sss (1,n)} \Big(\frac{1}{2 \log n}\Big)^{3\alpha},
\end{eq}
and therefore $i_{\sss (1,n)} \leq C (\log n)^{3\alpha}$. 
Further, it follows from \eqref{eq:d-2n} that $i_{\sss (2,n)} \leq K_2^{1/\alpha} (\log n)$. 
Therefore, the degree sequence in \eqref{defn:merged-degree} has length $n(1+o(1))$.

Since $i_{\sss (1,n)} \to \infty$, Assumption~\ref{assumption1}~(i) is satisfied by $(\bld{d}^{\sss (n)})_{n\geq 1}$. 
Also, for each fixed $K\geq 1$,  
\begin{eq}
	n^{-3\alpha}\sum_{i>K} (d_{i}^{\sss (n)})^3 
	\leq 
	\sum_{i>K} 8\theta_i^3+n^{-3\alpha}\sum_{i>K} (d_{i}^{\sss (2,n)})^3, 
\end{eq}
and thus \eqref{eqn:670} holds. 
Next, it can be easily checked that the remaining conditions in Assumption~\ref{assumption1}(ii) and \eqref{eq:defn-super-crit} hold for $(\bld{d}^{(n)})_{n\geq 1}$ by making use of the fact that $(\bld{d}^{\sss (2,n)})_{n\geq 1}$ satisfies Assumption~\ref{assumption1}(ii) and \eqref{eq:defn-super-crit}.

Finally we have to verify that $(\bld{d}^{\sss (n)})_{n\geq 1}$ satisfies Assumption~\ref{assumption-extra}. 
It suffices to show that there exist $C>1$ and $C'>0$ such that for all $n\geq 1$,
\[
\sum_i d_i^{\sss (n)}\ind{l< d_i^{\sss (n)}\leq Cl}
\geq
C' n/ l^{\tau-2}
\ \ \text{ for }\ \
1\leq l< d_1^{\sss (n)}\, .
\]
This can be proved in a straightforward way by using \eqref{eq:compactness}, \eqref{eq:d-2n}, and the definition of $\bld{d}^{\sss (n)}$ given in \eqref{defn:merged-degree}.
We omit the details.

\bibliographystyle{plain}
\bibliography{GHP-proof}

\end{document}